\documentclass{amsart}
\title{Categorifying the algebra of indexing systems}

\author{Jonathan Rubin}
\address{University of California Los Angeles,
Los Angeles, CA 90095}
\email{jrubin@math.ucla.edu}


\date{\today}

\renewcommand{\c}[1]{\mathcal{#1}}
\renewcommand{\t}[1]{\textnormal{#1}}
\renewcommand{\b}[1]{\mathbf{#1}}
\newcommand{\bb}[1]{\mathbb{#1}}
\newcommand{\vp}{\varphi}
\newcommand{\la}{\langle}
\newcommand{\ra}{\rangle}
\newcommand{\ul}[1]{\underline{#1}}
\newcommand{\ol}[1]{\overline{#1}}

\newcommand{\s}[1]{\mathscr{#1}}

\usepackage{amsmath,amssymb,amsthm,stackrel,tikz,enumitem,mathrsfs,stmaryrd,physics}

\theoremstyle{plain}
\newtheorem{lem}{Lemma}[section]
\newtheorem{thm}[lem]{Theorem}
\newtheorem*{thm*}{Theorem}
\newtheorem{cor}[lem]{Corollary}
\newtheorem{prop}[lem]{Proposition}
\newtheorem{exlem}[lem]{Example-Lemma}

\theoremstyle{remark}
\newtheorem{rem}[lem]{Remark}

\theoremstyle{definition}
\newtheorem{ex}[lem]{Example}

\newtheorem{conj}[lem]{Conjecture}
\newtheorem{defn}[lem]{Definition}

\theoremstyle{plain}
\newtheorem*{thmA}{Theorem A}
\newtheorem*{thmB}{Theorem B}

\newcommand{\cppa}{
	\begin{tikzpicture}[scale=0.15,baseline=0.2mm]
		\node(0) at (0,0) {$\cdot$};
		\node(1) at (0,1) {$\cdot$};
		\node(2) at (0,2) {$\cdot$};
		
	\end{tikzpicture}
}

\newcommand{\cppb}{
	\begin{tikzpicture}[scale=0.15,baseline=0.2mm]
		\node(0) at (0,0) {$\cdot$};
		\node(1) at (0,1) {$\cdot$};
		\node(2) at (0,2) {$\cdot$};
		
		\draw (0,0) -- (0,1);
	\end{tikzpicture}
}

\newcommand{\cppc}{
	\begin{tikzpicture}[scale=0.15,baseline=0.2mm]
		\node(0) at (0,0) {$\cdot$};
		\node(1) at (0,1) {$\cdot$};
		\node(2) at (0,2) {$\cdot$};
		
		\draw (0,0) -- (0,1);
		\draw plot [smooth, tension=1.5] coordinates {(0,0) (0.5,1) (0,2)};
	\end{tikzpicture}
}

\newcommand{\cppd}{
	\begin{tikzpicture}[scale=0.15,baseline=0.2mm]
		\node(0) at (0,0) {$\cdot$};
		\node(1) at (0,1) {$\cdot$};
		\node(2) at (0,2) {$\cdot$};
		
		\draw (0,1) -- (0,2);
	\end{tikzpicture}
}

\newcommand{\cppe}{
	\begin{tikzpicture}[scale=0.15,baseline=0.2mm]
		\node(0) at (0,0) {$\cdot$};
		\node(1) at (0,1) {$\cdot$};
		\node(2) at (0,2) {$\cdot$};
		
		\draw (0,0) -- (0,1);
		\draw plot [smooth, tension=1.5] coordinates {(0,0) (0.5,1) (0,2)};
		\draw (0,1) -- (0,2);
	\end{tikzpicture}
}

\newcommand{\cpppa}{
	\begin{tikzpicture}[scale=0.15,baseline=0.2mm]
		\node(0) at (0,0) {$\cdot$};
		\node(1) at (0,1) {$\cdot$};
		\node(2) at (0,2) {$\cdot$};
		\node(3) at (0,3) {$\cdot$};
		
	\end{tikzpicture}
}

\newcommand{\cpppb}{
	\begin{tikzpicture}[scale=0.15,baseline=0.2mm]
		\node(0) at (0,0) {$\cdot$};
		\node(1) at (0,1) {$\cdot$};
		\node(2) at (0,2) {$\cdot$};
		\node(3) at (0,3) {$\cdot$};
		
		\draw (0,0) -- (0,1);
	\end{tikzpicture}
}

\newcommand{\cpppc}{
	\begin{tikzpicture}[scale=0.15,baseline=0.2mm]
		\node(0) at (0,0) {$\cdot$};
		\node(1) at (0,1) {$\cdot$};
		\node(2) at (0,2) {$\cdot$};
		\node(3) at (0,3) {$\cdot$};
		
		\draw (0,0) -- (0,1);
		\draw plot [smooth, tension=1.5] coordinates {(0,0) (0.5,1) (0,2)};
	\end{tikzpicture}
}

\newcommand{\cpppf}{
	\begin{tikzpicture}[scale=0.15,baseline=0.2mm]
		\node(0) at (0,0) {$\cdot$};
		\node(1) at (0,1) {$\cdot$};
		\node(2) at (0,2) {$\cdot$};
		\node(3) at (0,3) {$\cdot$};
		
		\draw (0,0) -- (0,1);
		\draw plot [smooth, tension=1.5] coordinates {(0,0) (0.5,1) (0,2)};
		\draw plot [smooth, tension=1.5] coordinates {(0,0) (1,1.5) (0,3)};
	\end{tikzpicture}
}

\newcommand{\cpppg}{
	\begin{tikzpicture}[scale=0.15,baseline=0.2mm]
		\node(0) at (0,0) {$\cdot$};
		\node(1) at (0,1) {$\cdot$};
		\node(2) at (0,2) {$\cdot$};
		\node(3) at (0,3) {$\cdot$};
		
		\draw (0,0) -- (0,1);
		\draw plot [smooth, tension=1.5] coordinates {(0,0) (0.5,1) (0,2)};
		\draw (0,1) -- (0,2);
		\draw plot [smooth, tension=1.5] coordinates {(0,0) (1,1.5) (0,3)};
	\end{tikzpicture}
}

\newcommand{\cppph}{
	\begin{tikzpicture}[scale=0.15,baseline=0.2mm]
		\node(0) at (0,0) {$\cdot$};
		\node(1) at (0,1) {$\cdot$};
		\node(2) at (0,2) {$\cdot$};
		\node(3) at (0,3) {$\cdot$};
		
		\draw (0,1) -- (0,2);
		\draw plot [smooth, tension=1.5] coordinates {(0,1) (-0.5,2) (0,3)};
	\end{tikzpicture}
}

\newcommand{\cpppi}{
	\begin{tikzpicture}[scale=0.15,baseline=0.2mm]
		\node(0) at (0,0) {$\cdot$};
		\node(1) at (0,1) {$\cdot$};
		\node(2) at (0,2) {$\cdot$};
		\node(3) at (0,3) {$\cdot$};
		
		\draw (0,0) -- (0,1);
		\draw plot [smooth, tension=1.5] coordinates {(0,0) (0.5,1) (0,2)};
		\draw (0,1) -- (0,2);
		\draw plot [smooth, tension=1.5] coordinates {(0,0) (1,1.5) (0,3)};
		\draw plot [smooth, tension=1.5] coordinates {(0,1) (-0.5,2) (0,3)};
	\end{tikzpicture}
}

\newcommand{\cpppl}{
	\begin{tikzpicture}[scale=0.15,baseline=0.2mm]
		\node(0) at (0,0) {$\cdot$};
		\node(1) at (0,1) {$\cdot$};
		\node(2) at (0,2) {$\cdot$};
		\node(3) at (0,3) {$\cdot$};
		
		\draw (0,0) -- (0,1);
		\draw (0,2) -- (0,3);
	\end{tikzpicture}
}

\newcommand{\cpppm}{
	\begin{tikzpicture}[scale=0.15,baseline=0.2mm]
		\node(0) at (0,0) {$\cdot$};
		\node(1) at (0,1) {$\cdot$};
		\node(2) at (0,2) {$\cdot$};
		\node(3) at (0,3) {$\cdot$};
		
		\draw (0,0) -- (0,1);
		\draw plot [smooth, tension=1.5] coordinates {(0,0) (0.5,1) (0,2)};
		\draw plot [smooth, tension=1.5] coordinates {(0,0) (1,1.5) (0,3)};
		\draw (0,2) -- (0,3);
	\end{tikzpicture}
}

\newcommand{\cpppn}{
	\begin{tikzpicture}[scale=0.15,baseline=0.2mm]
		\node(0) at (0,0) {$\cdot$};
		\node(1) at (0,1) {$\cdot$};
		\node(2) at (0,2) {$\cdot$};
		\node(3) at (0,3) {$\cdot$};
		
		\draw (0,0) -- (0,1);
		\draw plot [smooth, tension=1.5] coordinates {(0,0) (0.5,1) (0,2)};
		\draw (0,1) -- (0,2);
		\draw plot [smooth, tension=1.5] coordinates {(0,0) (1,1.5) (0,3)};
		\draw plot [smooth, tension=1.5] coordinates {(0,1) (-0.5,2) (0,3)};
		\draw (0,2) -- (0,3);
	\end{tikzpicture}
}

\newcommand{\kindt}{
	\begin{tikzpicture}[scale=0.12,baseline=0.2mm]
		\node(S) at (0,0) {$\cdot$};
		\node(W) at (-1.5,1) {$\cdot$};
		\node(C) at (0,1) {$\cdot$};
		\node(E) at (1.5,1) {$\cdot$};
		\node(N) at (0,2) {$\cdot$};
	\end{tikzpicture}
}

\newcommand{\kindb}{
	\begin{tikzpicture}[scale=0.12,baseline=0.2mm]
		\node(S) at (0,0) {$\cdot$};
		\node(W) at (-1.5,1) {$\cdot$};
		\node(C) at (0,1) {$\cdot$};
		\node(E) at (1.5,1) {$\cdot$};
		\node(N) at (0,2) {$\cdot$};
		
		\draw (0,0) -- (0,1) ;
	\end{tikzpicture}
}

\newcommand{\kindab}{
	\begin{tikzpicture}[scale=0.12,baseline=0.2mm]
		\node(S) at (0,0) {$\cdot$};
		\node(W) at (-1.5,1) {$\cdot$};
		\node(C) at (0,1) {$\cdot$};
		\node(E) at (1.5,1) {$\cdot$};
		\node(N) at (0,2) {$\cdot$};
		
		\draw (0,0) -- (-1.5,1) ;
		\draw (0,0) -- (0,1) ;
	\end{tikzpicture}
}

\newcommand{\kindabc}{
	\begin{tikzpicture}[scale=0.12,baseline=0.2mm]
		\node(S) at (0,0) {$\cdot$};
		\node(W) at (-1.5,1) {$\cdot$};
		\node(C) at (0,1) {$\cdot$};
		\node(E) at (1.5,1) {$\cdot$};
		\node(N) at (0,2) {$\cdot$};
		
		\draw (0,0) -- (-1.5,1) ;
		\draw (0,0) -- (0,1) ;
		\draw (0,0) -- (1.5,1) ;
	\end{tikzpicture}
}

\newcommand{\kindkb}{
	\begin{tikzpicture}[scale=0.12,baseline=0.2mm]
		\node(S) at (0,0) {$\cdot$};
		\node(W) at (-1.5,1) {$\cdot$};
		\node(C) at (0,1) {$\cdot$};
		\node(E) at (1.5,1) {$\cdot$};
		\node(N) at (0,2) {$\cdot$};
		
		\draw (0,0) -- (-1.5,1) ;
		\draw (0,0) -- (1.5,1) ;
		\draw (0,1) -- (0,2) ;
	\end{tikzpicture}
}

\newcommand{\kindkc}{
	\begin{tikzpicture}[scale=0.12,baseline=0.2mm]
		\node(S) at (0,0) {$\cdot$};
		\node(W) at (-1.5,1) {$\cdot$};
		\node(C) at (0,1) {$\cdot$};
		\node(E) at (1.5,1) {$\cdot$};
		\node(N) at (0,2) {$\cdot$};
		
		\draw (0,0) -- (-1.5,1) ;
		\draw (0,0) -- (0,1) ;
		\draw (1.5,1) -- (0,2) ;
	\end{tikzpicture}
}

\newcommand{\kindkt}{
	\begin{tikzpicture}[scale=0.12,baseline=0.2mm]
		\node(S) at (0,0) {$\cdot$};
		\node(W) at (-1.5,1) {$\cdot$};
		\node(C) at (0,1) {$\cdot$};
		\node(E) at (1.5,1) {$\cdot$};
		\node(N) at (0,2) {$\cdot$};
		
		\draw (0,0) -- (-1.5,1) ;
		\draw (0,0) -- (0,1) ;
		\draw (0,0) -- (1.5,1) ;
		\draw plot [smooth, tension=1.5] coordinates {(0,0) (-0.5,1) (0,2)};
	\end{tikzpicture}
}

\newcommand{\kindkat}{
	\begin{tikzpicture}[scale=0.12,baseline=0.2mm]
		\node(S) at (0,0) {$\cdot$};
		\node(W) at (-1.5,1) {$\cdot$};
		\node(C) at (0,1) {$\cdot$};
		\node(E) at (1.5,1) {$\cdot$};
		\node(N) at (0,2) {$\cdot$};
		
		\draw (0,0) -- (-1.5,1) ;
		\draw (0,0) -- (0,1) ;
		\draw (0,0) -- (1.5,1) ;
		\draw (-1.5,1) -- (0,2) ;
		\draw plot [smooth, tension=1.5] coordinates {(0,0) (-0.5,1) (0,2)};	
	\end{tikzpicture}
}

\newcommand{\kindkbt}{
	\begin{tikzpicture}[scale=0.12,baseline=0.2mm]
		\node(S) at (0,0) {$\cdot$};
		\node(W) at (-1.5,1) {$\cdot$};
		\node(C) at (0,1) {$\cdot$};
		\node(E) at (1.5,1) {$\cdot$};
		\node(N) at (0,2) {$\cdot$};
		
		\draw (0,0) -- (-1.5,1) ;
		\draw (0,0) -- (0,1) ;
		\draw (0,0) -- (1.5,1) ;
		\draw (0,1) -- (0,2) ;
		\draw plot [smooth, tension=1.5] coordinates {(0,0) (-0.5,1) (0,2)};
	\end{tikzpicture}
}

\newcommand{\kindkct}{
	\begin{tikzpicture}[scale=0.12,baseline=0.2mm]
		\node(S) at (0,0) {$\cdot$};
		\node(W) at (-1.5,1) {$\cdot$};
		\node(C) at (0,1) {$\cdot$};
		\node(E) at (1.5,1) {$\cdot$};
		\node(N) at (0,2) {$\cdot$};
		
		\draw (0,0) -- (-1.5,1) ;
		\draw (0,0) -- (0,1) ;
		\draw (0,0) -- (1.5,1) ;
		\draw (1.5,1) -- (0,2) ;
		\draw plot [smooth, tension=1.5] coordinates {(0,0) (-0.5,1) (0,2)};
	\end{tikzpicture}
}

\newcommand{\kindkac}{
	\begin{tikzpicture}[scale=0.12,baseline=0.2mm]
		\node(S) at (0,0) {$\cdot$};
		\node(W) at (-1.5,1) {$\cdot$};
		\node(C) at (0,1) {$\cdot$};
		\node(E) at (1.5,1) {$\cdot$};
		\node(N) at (0,2) {$\cdot$};
		
		\draw (0,0) -- (-1.5,1) ;
		\draw (0,0) -- (0,1) ;
		\draw (0,0) -- (1.5,1) ;
		\draw (-1.5,1) -- (0,2) ;
		\draw (1.5,1) -- (0,2) ;
		\draw plot [smooth, tension=1.5] coordinates {(0,0) (-0.5,1) (0,2)};
	\end{tikzpicture}
}

\newcommand{\dsia}{
	\begin{tikzpicture}[scale=0.2,baseline=0.2mm]
		\node(S) at (0,0) {$\cdot$};
		\node(W) at (-1.5,1) {$\cdot$};
		\node(C) at (0,1) {$\cdot$};
		\node(E) at (1.5,1) {$\cdot$};
		\node(N) at (0,2) {$\cdot$};
		\node(3) at (0.6,1.2) {$\cdot$};
		
	\end{tikzpicture}
}

\newcommand{\dsib}{
	\begin{tikzpicture}[scale=0.2,baseline=0.2mm]
		\node(S) at (0,0) {$\cdot$};
		\node(W) at (-1.5,1) {$\cdot$};
		\node(C) at (0,1) {$\cdot$};
		\node(E) at (1.5,1) {$\cdot$};
		\node(N) at (0,2) {$\cdot$};
		\node(3) at (0.6,1.2) {$\cdot$};
		
		\draw (0,0) -- (-1.5,1) ;
		\draw (0,0) -- (0,1) ;
		\draw (0,0) -- (1.5,1) ;
		
	\end{tikzpicture}
}

\newcommand{\dsic}{
	\begin{tikzpicture}[scale=0.2,baseline=0.2mm]
		\node(S) at (0,0) {$\cdot$};
		\node(W) at (-1.5,1) {$\cdot$};
		\node(C) at (0,1) {$\cdot$};
		\node(E) at (1.5,1) {$\cdot$};
		\node(N) at (0,2) {$\cdot$};
		\node(3) at (0.6,1.2) {$\cdot$};
		
		\draw (0,0) -- (0.6,1.2) ;
		
	\end{tikzpicture}
}

\newcommand{\dsii}{
	\begin{tikzpicture}[scale=0.2,baseline=0.2mm]
		\node(S) at (0,0) {$\cdot$};
		\node(W) at (-1.5,1) {$\cdot$};
		\node(C) at (0,1) {$\cdot$};
		\node(E) at (1.5,1) {$\cdot$};
		\node(N) at (0,2) {$\cdot$};
		\node(3) at (0.6,1.2) {$\cdot$};
		
		\draw (0,0) -- (-1.5,1) ;
		\draw (0,0) -- (0,1) ;
		\draw (0,0) -- (1.5,1) ;
		
		\draw (0,0) -- (0.6,1.2) ;
		
		\draw plot [smooth, tension=1.5] coordinates {(0,0) (-0.5,1) (0,2)};
		\draw (-1.5,1) -- (0,2) ;
		\draw (0,1) -- (0,2) ;
		\draw (1.5,1) -- (0,2) ;
		\draw (0.6,1.2) -- (0,2) ;
	\end{tikzpicture}
}

\begin{document}
\maketitle

\begin{abstract} The homotopy category of $N_\infty$ operads is equivalent to a finite lattice, and as the ambient group varies, there are various image constructions between these lattices. In this paper, we explain how to lift this algebraic structure back to the operad level. We show that lattice joins and meets correspond to derived operadic coproducts and products, and we show that the image constructions correspond to derived operadic induction, restriction, and coinduction, at least when taken along an injective homomorphism.

We also prove that a derived variant of the Boardman-Vogt tensor product lifts the join. Our result does not resolve Blumberg and Hill's conjecture on the usual tensor product, but it does imply that every $N_\infty$ ring spectrum can be replaced with an equivalent spectrum, which is equipped with a self-interchanging operad action.
\end{abstract}

\tableofcontents

\section{Introduction}\label{sec:intro}

Transfer and norm maps are defining features of equivariant stable homotopy theory. From a classical standpoint, they arise geometrically, but in more modern terms, they arise from actions of $N_\infty$ operads on spaces and spectra. Broadly speaking, such operads represent equivariant enhancements of homotopy commutative monoid structures. They include Steiner and linear isometries operads, which parametrize additive and multiplicative structures on spectra over incomplete universes, but they are strictly more general. Nevertheless, $N_\infty$ algebras are quite natural from an algebraic standpoint. Localizations of equivariant commutative ring spectra are generally $N_\infty$ algebras (cf. \cite{HH} and \cite{White}), and recent work of Blumberg and Hill \cite{BHOspec} shows how to build various incomplete equivariant stable categories from various categories of $N_\infty$ spaces. 

The study of $N_\infty$ operads and algebras was initiated in \cite{BH}. In this paper, Blumberg and Hill laid much of the foundations for the subject, and they also made two conjectures. The first \cite[p. 4 and Conjecture 5.11]{BH} concerned the classification of $N_\infty$ operads. Over the course of their analysis, Blumberg and Hill proved that the homotopy category $\t{Ho}(N_\infty\t{-}\b{Op}^G)$ of $N_\infty$ $G$-operads embeds fully and faithfully into a combinatorially-defined lattice $\b{Ind}(G)$ of $G$-indexing systems. They conjectured that this embedding was an equivalence, and this was subsequently proven in \cite{BonPer}, \cite{GutWhite}, and \cite{RubComb}.

The second conjecture \cite[Conjecture 6.27]{BH} concerned the lattice structure of $\b{Ind}(G)$. It is straightforward to show that products of $N_\infty$ operads correspond to meets of indexing systems under the equivalence $\t{Ho}(N_\infty\t{-}\b{Op}^G) \simeq \b{Ind}(G)$. In analogy to the Dunn additivity theorem \cite{Dunn}, Blumberg and Hill conjectured that Boardman-Vogt tensor products of suitably cofibrant $N_\infty$ operads correspond to joins. This remains an open problem.

Our present paper grew out of attempts to resolve the second conjecture, and also to understand how other algebraic operations on the level indexing systems translate into topological constructions on the level of $N_\infty$ operads. We have in mind the lattice structure on $\b{Ind}(G)$ for individual finite groups $G$, and also the analogues to induction, restriction, and coinduction as $G$ varies. Part of this work was already done in \cite{BH}. As mentioned earlier, Blumberg and Hill showed that products of $N_\infty$ operads correspond to lattice meets, and they also identified the indexing system associated to a coinduced $N_\infty$ operad \cite[\S6.2]{BH}.

The dual problems are trickier. It is not obvious what a coproduct of $N_\infty$ operads or an induced $N_\infty$ operad even should be, because the usual operadic constructions do not have the right homotopical properties. One could imagine modifying the standard topological constructions, but we take a different approach. As explained in \cite{RubComb}, the homotopy theory of $N_\infty$ operads can be modeled using discrete operads, and it is easy to make sense of coproducts and induction in that setting. Thus, we analyze how the algebra of indexing systems lifts to combinatorial operads, and then we translate things into topology at the end (cf. \S\ref{subsec:topprodcopten} and \S\ref{subsec:topinterpindrescoind}).

In summary, we prove that joins in $\b{Ind}(G)$ lift to derived operadic coproducts and tensor products, and we define purely algebraic versions of induction, restriction, and coinduction that correspond to derived operadic induction, restriction, and coinduction. Our constructions are naturally performed on the level of discrete operads, but it is straightforward to read off the corresponding topological constructions from them. Ordinary products, restrictions, and coinductions of $N_\infty$ operads are already homotopically correct, but derived $N_\infty$ coproducts and inductions are quite different from the standard constructions.

We now state our results more precisely. The simplest discrete models for $N_\infty$ operads are operads in $G$-sets, which are $\Sigma$-free and have $G$-fixed operations of all arities. We call these objects $N$ operads, and we write $A(\s{O})$ for the indexing system associated to any such operad $\s{O}$. The following result combines \cite[Proposition 5.1]{BH} and Theorems \ref{thm:Ncoprjoin} and \ref{thm:NopBV}.

\begin{thmA} Suppose $\s{O}$ and $\s{P}$ are $N$ operads. Then:
	\begin{enumerate}
		\item{}the product $\s{O} \times \s{P}$ is a $N$ operad and $A(\s{O} \times \s{P}) = A(\s{O}) \land A(\s{P})$,
		\item{}the coproduct $\s{O} * \s{P}$ is a $N$ operad and $A(\s{O} * \s{P}) = A(\s{O}) \lor A(\s{P})$, and
		\item{}if $\s{O}$ and $\s{P}$ are both retracts of free operads, then $\s{O} \otimes \s{P}$ is a $N$ operad and $A(\s{O} \otimes \s{P}) = A(\s{O}) \lor A(\s{P})$.
	\end{enumerate}
\end{thmA}

Part (3) is a precise combinatorial analogue to \cite[Conjecture 6.27]{BH}, but it does not imply the topological result. The $N_\infty$ operad associated to $\s{O}$ is obtained by attaching cells to make all fixed-point subspaces of $\s{O}$ contractible, and this construction does not preserve colimits. Thus, we can only deduce part of the conjecture from the combinatorial result (cf. Proposition \ref{prop:combtentop}). However, we do conclude that every $N_\infty$ ring spectrum can be equipped with a self-interchanging operad action, at least up to equivalence  (Theorem \ref{thm:Ninftyinterchange}).

The results for induction, restriction, and coinduction require a bit more setup, because $N$ operads do not induce up to $N$ operads. Instead, we use marked $G$-operads, by which we mean operads $\s{O}$ in $G$-sets, equipped with a chosen unit $u \in \s{O}(0)^G$ and product $p \in \s{O}(2)^G$. The $N_\infty$ operad corresponding to a marked operad $\s{O}$ is obtained by taking a $\Sigma$-free, free resolution of $\s{O}$, and then attaching cells just as before. For any homomorphism $f : G \to G'$ between finite groups, pulling back and Kan extending along $f$ defines change-of-group adjunctions $\t{ind}_f \dashv \t{res}_f \dashv \t{coind}_f$ for marked operads. The adjunction $\t{res}_f \dashv \t{coind}_f$ always derives, and the adjunction $\t{ind}_f \dashv \t{res}_f$ derives if $f$ is injective.

There are direct algebraic analogues to these operadic constructions, but they are most easily defined using the transfer system formalism from \cite{RubChar}. Informally, a transfer system $\to$ is the set of orbits in an indexing system. The two notions are logically equivalent, but transfer systems are easier to manage because they are smaller. We write $\to_{\s{O}}$ for the transfer system associated to an operad $\s{O}$. Every homomorphism $f : G \to G'$ between finite groups gives rise to change of group adjunctions $f_L \dashv f^{-1}_R$ and $f^{-1}_L \dashv f_R$ for transfer systems, and $f^{-1}_R = f^{-1}_L$ if $f$ is injective. The next result is Theorem \ref{thm:iminvimcompatindrescoind}.

\begin{thmB} Suppose $f : G \to G'$ is an arbitrary homomorphism between finite groups. Then:
	\begin{enumerate}
		\item{} $\to_{\bb{L}\t{res}_f \s{O}'} \,\, = f^{-1}_L(\to_{\s{O}'})$ for every marked $G'$-operad $\s{O}'$,
		\item{} $\to_{ \bb{R}\t{coind}_f \s{O}} \,\, =  f_R(\to_{\s{O}}) $ for every marked $G$-operad $\s{O}$, and
		\item{} if $f$ is injective, then $\to_{\bb{L}\t{ind}_f \s{O}} \,\, = f_L(\to_{\s{O}})$ for every marked $G$-operad $\s{O}$.
	\end{enumerate}
\end{thmB}

There are analogues to parts (1) and (2) of Theorem A for marked operads, but they are easier (cf. \S\ref{subsec:markoplat}). There are also analogues to parts (1) and (2) of Theorem B for $N$ operads, but they follow from the results for marked operads (cf. \S\ref{subsec:Nopindrescoind}).

The constructions in this paper are quite explicit, and we have tried to give examples whenever possible. Moreover, the correspondences in Theorems A and B have already been applied by Balchin, Barnes, and Roitzheim to interpret their decomposition of the lattice of $C_{p^n}$-transfer systems on the operad level \cite[Remark 1]{BBR}. We hope to see further concrete applications.

\subsection*{Organization} This paper uses a handful of ideas from \cite{RubComb} and \cite{RubChar}, so we review the relevant machinery in \S\ref{sec:background}. After that, we get down to work. In \S\ref{sec:meetjoin}, we give a quick description of how meets and joins of transfer systems are calculated, and then we lift these lattice operations to the operad level in \S\ref{sec:prodcopten}. Similarly, we introduce image and inverse image constructions for transfer systems in \S\ref{sec:iminvim}, and then we lift these constructions to the operad level in \S\ref{sec:indrescoind}. Appendix \ref{app:quotop} describes a method for identifying quotient operads. It contains the most technical details needed for the proofs in \S\ref{sec:prodcopten}.

\subsection*{Acknowledgements} The first half of this paper is based on work from my dissertation, and it is a pleasure to thank Peter May for his sage advice and support. The second half of this paper grew out of conversations with Mike Hill, and it is a pleasure to thank him for continued guidance and inspiration. This work was partially supported by NSF Grant DMS--1803426.

\section{Combinatorial $N_\infty$ operads}\label{sec:background}

We review some preliminaries in this section, with an emphasis on the ways in which the homotopy theory of $N_\infty$ operads is algebraic. In \S\ref{subsec:backts}, we summarize the classification of $N_\infty$ operads in terms of transfer systems and indexing systems, and in \S\ref{subsec:combops}, we recall some basic properties of $N$ operads and marked $G$-operads.

\subsection{Transfer systems and indexing systems}\label{subsec:backts} Let $G$ be a finite group, and let $\s{O}$ be an operad in the category $\b{Top}^G$ of compactly generated weak Hausdorff $G$-spaces. An $\s{O}$-action is a continuous and equivariant parametrization of operations by $\s{O}$. The stabilizers of points $f \in \s{O}$ in the operad determine how commutative and equivariant the corresponding operations are, and the topology of $\s{O}$ imposes homotopy relations between these operations. A $N_\infty$ operad parametrizes a particular kind of equivariant homotopy-commutative structure. 

\begin{defn}A \emph{$N_\infty$ operad} is a operad $\s{O}$ in $\b{Top}^G$ such that
	\begin{enumerate}
		\item{}for every $n \geq 0$, the space $\s{O}(n)$ is $\Sigma_n$-free,
		\item{}for every $n \geq 0$ and subgroup $\Gamma \subset G \times \Sigma_n$, the space $\s{O}(n)^\Gamma$ is either empty or contractible, and
		\item{}for every $n \geq 0$, the space $\s{O}(n)^G$ is contractible.
	\end{enumerate}
\end{defn}

The first condition ensures that $\s{O}$ parametrizes no strict commutativity relations, and the third condition ensures that $\s{O}$ parametrizes a homotopy coherent commutative monoid structure, in which all data is $G$-equivariant. The $\Gamma$-fixed points give rise to equivariant transfers.

More explicitly, suppose $K \subset H \subset G$ is a chain of subgroups with $\abs{H:K} = n$, and suppose $\sigma : H \to \Sigma_n$ is a permutation representation of the $H$-orbit $H/K$. Let $\Gamma(H/K) = \{ (h, \sigma(h)) \, | \, h \in H \} \subset G \times \Sigma_n$ be the graph of $\sigma$, and assume that $f \in \s{O}(n)^{\Gamma(H/K)}$. Then $f$ represents $G$-maps
	\[
	\ol{\t{tr}}_K^H : G \times_H \t{coind}^H_K\t{res}^G_K X \to X	\quad\t{and}\quad	\ol{\t{n}}_K^H : G_+ \land_H N_K^H\t{res}^G_K E \to E
	\]
on all $\s{O}$-algebra $G$-spaces $X$ and $G$-spectra $E$. These maps are external transfers and norms. In the first case, passing to the adjoint $\t{coind}_K^H\t{res}^G_K X \to \t{res}^G_H X$ and then taking $H$-fixed points yields an internal transfer map $X^K \to X^H$.

Given any $N_\infty$ $G$-operad $\s{O}$, there is a corresponding relation on the set $\b{Sub}(G)$ of all subgroups $G$, which encodes the transfers parametrized by $\s{O}$.

\begin{defn}\label{defn:tsNinfty} Suppose $\s{O}$ is a $N_\infty$ $G$-operad. Define a binary relation $\to_{\s{O}}$ on $\b{Sub}(G)$ by
	\[
	K \to_{\s{O}} H	\quad\t{if and only if}\quad		K \subset H \t{ and } \s{O}(\abs{H:K})^{\Gamma(H/K)} \neq \varnothing,
	\]
where $\Gamma(H/K)$ is the graph of some chosen permutation representation of $H/K$.
\end{defn}

The relation $\to_{\s{O}}$ satisfies conditions that reflect the operad structure on $\s{O}$. These conditions are axiomatized in the next definition, formulated independently in \cite{BBR} and \cite{RubChar}.

\begin{defn} A \emph{$G$-transfer system} is a partial order $\to$ on $\b{Sub}(G)$ such that for any $K, H \in \b{Sub}(G)$, if the relation $K \to H$ holds, then:
	\begin{enumerate}[label=(\alph*)]
		\item{}the inclusion $K \subset H$ holds,
		\item{}the relation $gKg^{-1} \to gHg^{-1}$ holds for every $g \in G$, and
		\item{}the relation $L \cap K \to L$ holds for every subgroup $L \subset H$.
	\end{enumerate}
We let $\b{Tr}(G)$ denote the lattice of all $G$-transfer systems, ordered under refinement.
\end{defn}

Succinctly, a $G$-transfer system is a partial order on $\b{Sub}(G)$ that refines inclusion and is closed under conjugation and restriction. We identify a $G$-transfer system $\to$ with the set of pairs $\{ (K,H) \in \b{Sub}(G)^{\times 2} \, | \, K \to H\}$, and we visualize $\to$ as a graph, whose nodes are the subgroups of $G$, and whose edges represent nontrivial relations in $\to$. 

More generally, if $\s{O}$ is a $N_\infty$ $G$-operad and $f \in \s{O}(n)$, then the stabilizer $\t{Stab}(f)$ is a graph subgroup in the following sense.

\begin{defn}Suppose $n \geq 0$ and $\Gamma \subset G \times \Sigma_n$ is a subgroup. Then $\Gamma$ is a \emph{graph subgroup} if there is a subgroup $H \subset G$ and a $n$-element $H$-set $T$ such that $\Gamma = \{ (h, \sigma(h)) \, | \, h \in H \}$ for some permutation representation $\sigma : H \to \Sigma_n$ of $T$. In such a case, we write $\Gamma = \Gamma(T)$.
\end{defn}

If the operation $f \in \s{O}$ satisfies $\t{Stab}(f) = \Gamma(T)$, then $f$ represents an external $T$-indexed transfer or norm on $G$-spaces and $G$-spectra. It is sometimes convenient to keep track of all such $T$-indexed operations.

\begin{defn}\label{defn:Tadm} Suppose $\s{O}$ is a $N_\infty$ $G$-operad. For any subgroup $H \subset G$ and finite $H$-set $T$, we say $T$ is \emph{admissible} for $\s{O}$ if $\s{O}(\abs{T})^{\Gamma(T)} \neq \varnothing$. We write $A(\s{O})$ for the $\b{Sub}(G)$-graded class of all admissible sets of $\s{O}$.
\end{defn}

The class $A(\s{O})$ also satisfies conditions that reflect the operad structure on $\s{O}$.

\begin{defn}A \emph{$G$-indexing system} is a $\b{Sub}(G)$-graded class $\c{I}$, whose $H$-component $\c{I}(H)$ is class of finite $H$-sets that contains all trivial actions, and which is closed under isomorphism, conjugation, restriction, subobjects, coproducts, and self-induction, i.e. if $T \in \c{I}(K)$ and $H/K \in \c{I}(H)$, then $\t{ind}_K^H T = H \times_K T \in \c{I}(H)$. We write $\b{Ind}(G)$ for the lattice of all $G$-transfer systems, ordered under inclusion.
\end{defn}

Every $G$-indexing system $\c{I}$ determines a $G$-transfer system $\to_{\c{I}}$, where
	\[
	K \to_{\c{I}} H	\quad\t{if and only if}\quad	K \subset H \t{ and } H/K \in \c{I}(H).
	\]
Moreover, indexing systems and transfer systems are equivalent, essentially because indexing systems are determined by their orbits.

\begin{thm}[\cite{BBR} and \cite{RubChar}] The map $\to_{\bullet} \, : \b{Ind}(G) \to \b{Tr}(G)$ is a lattice isomorphism for any finite group $G$.
\end{thm}

Every $N_\infty$ $G$-operad $\s{O}$ gives rise to a $G$-indexing system $A(\s{O})$ and a $G$-transfer system $\to_{\s{O}}$, which are related by the formula $\to_{A(\s{O})} \,\, = \,\, \to_{\s{O}}$. Furthermore, these objects completely determine $\s{O}$ up to homotopy. Declare a map $\vp : \s{O} \to \s{P}$ between $N_\infty$ operads to be a \emph{weak equivalence} if the map $\vp : \s{O}(n)^{\Gamma} \to \s{P}(n)^{\Gamma}$ is a weak homotopy equivalence of spaces for every integer $n \geq 0$ and graph subgroup $\Gamma \subset G \times \Sigma_n$. Let $\t{Ho}(N_\infty\t{-}\b{Op}^G)$ denote the category of $N_\infty$ $G$-operads with weak equivalences inverted. Then we have the following classification theorem.

\begin{thm}[\cite{BH}, \cite{BonPer}, \cite{GutWhite}, and \cite{RubComb}]\label{thm:classify} The functor
	\[
	N_\infty\t{-}\b{Op}^G \stackrel{A}{\longrightarrow} \b{Ind}(G) \cong \b{Tr}(G)
	\]
that sends an operad $\s{O}$ to the indexing system $A(\s{O})$, and an operad map $\s{O} \to \s{P}$ to the inclusion $A(\s{O}) \subset A(\s{P})$ induces an equivalence $\t{Ho}(N_\infty\t{-}\b{Op}^G) \simeq \b{Ind}(G) \cong \b{Tr}(G)$ of $1$-categories.
\end{thm}

In fact, the mapping spaces in the hammock localization $L^H(N_\infty\t{-}\b{Op}^G)$ are all either empty or contractible \cite[Proposition 5.5]{BH}, so $\to_{\bullet}$ also induces a DK equivalence $L^H(N_\infty\t{-}\b{Op}^G) \simeq \b{Ind}(G) \cong \b{Tr}(G)$. This is one sense in which the homotopy theory of $N_\infty$ operads is algebraic.

\subsection{Combinatorial models of $N_\infty$ operads}\label{subsec:combops}

Another sense in which the homotopy theory of $N_\infty$ operads is algebraic is that there are categories of discrete $G$-operads, whose underlying homotopy theories are equivalent to $L^H(N_\infty\t{-}\b{Op}^G)$. We review some material from \cite[\S\S3--4]{RubComb}.

Let $\b{Op}^G$ denote the category of symmetric operads in $\b{Set}^G$. The simplest discrete models for $N_\infty$ operads are their natural analogues in $\b{Op}^G$.

\begin{defn} A \emph{$N$ operad} is an operad $\s{O}$ in $\b{Set}^G$ such that
	\begin{enumerate}
		\item{}for every $n \geq 0$, the set $\s{O}(n)$ is $\Sigma_n$-free, and
		\item{}for every $n \geq 0$, the set $\s{O}(n)^G$ is nonempty.
	\end{enumerate}
We let $N\t{-}\b{Op}^G$ denote the full subcategory of $\b{Op}^G$ spanned by the $N$ operads.
\end{defn}

We construct $N_\infty$ operads from $N$ operads by attaching cells. Let $\b{sSet}$ denote the category of simplicial sets, and let
	\[
	(-)_0 : \b{sSet} \leftrightarrows \b{Set} : E
	\]
be the $0$-simplices functor and its right adjoint. For any set $X$ and $q \geq 0$, the set of $q$-simplices of $EX$ is $X^{\times q + 1}$, and the face and degeneracy maps of $EX$ are obtained by omitting and repeating coordinates. The simplicial set $EX$ is contractible whenever $X$ is nonempty, and we have $E\varnothing = \varnothing$. 

The functor $E$ and geometric realization $\abs{\cdot}$ both preserve finite limits, and therefore we obtain a composite functor
	\[
	\bb{E} = \abs{\cdot} \circ E : N\t{-}\b{Op}^G \to N_\infty\t{-}\b{Op}^G.
	\]
Declare a morphism $\vp : \s{O} \to \s{P}$ in $\b{Op}^G$ to be a \emph{weak equivalence} if the induced map $\abs{E\vp} : \abs{E\s{O}(n)}^\Gamma \to \abs{E\s{P}(n)}^\Gamma$ is a weak homotopy equivalence of spaces for every integer $n \geq 0$ and graph subgroup $\Gamma \subset G \times \Sigma_n$. This boils down to the condition that $\s{O}(n)^\Gamma$ is nonempty whenever $\s{P}(n)^\Gamma$ is. Then the functor $\bb{E}$ preserves weak equivalences, and we actually obtain an equivalence of homotopy theories.

\begin{thm}[\protect{\cite[Theorem 3.6]{RubComb}}]\label{thm:Nopmod} The functor
	\[
	\bb{E} : N\t{-}\b{Op}^G \to N_\infty\t{-}\b{Op}^G
	\]
induces a DK equivalence between the hammock localizations of $N\t{-}\b{Op}^G$ and $N_\infty\t{-}\b{Op}^G$.
\end{thm}

The category $N\t{-}\b{Op}^G$ is simple and explicit, but it has a number of deficiencies. In particular, $N\t{-}\b{Op}^G$ is neither complete nor cocomplete, and operadic induction $\t{ind}_H^G : \b{Op}^H \to \b{Op}^G$ does not preserve $N$ operads. We introduce a model category of operads in $\b{Set}^G$ to remedy these issues.

To start, note that the category $\b{Op}^G$ is complete and cocomplete for formal reasons. We write $\s{O} * \s{P}$ for the coproduct in $\b{Op}^G$ in analogy to the coproduct of nonabelian groups. The category $\b{Op}^G$ is also locally finitely presentable. If
	\[
	F : \b{Sym}^G \leftrightarrows \b{Op}^G : U
	\]
is the free-forgetful adjunction from symmetric sequences of $G$-sets, then the free operads $F(G \times \Sigma_n)$ form a strong generator for $\b{Op}^G$ in the sense of \cite{AR}.

Let $\b{F}$ be the free $G$-operad on the symmetric sequence $\frac{G \times \Sigma_0}{G} \sqcup \frac{G \times \Sigma_2}{G}$, and write $\b{Op}^G_+$ for the slice category $\b{F}/\b{Op}^G$ of $G$-operads under $\b{F}$. By adjunction, an object of $\b{Op}^G_+$ is an operad $\s{O} \in \b{Op}^G$ equipped with a marked constant $u \in \s{O}(0)^G$ and binary product $p\in \s{O}(2)^G$, and a morphism in $\b{Op}^G_+$ is a morphism of $G$-operads that preserves the marked operations. The category $\b{Op}^G_+$ is also complete, cocomplete, and locally finitely presentable.

The category $\b{Op}^G_+$ carries a simplicial enrichment, which is most quickly defined using the adjunction $(-)_0 \dashv E$ from above.  Both of the functors $(-)_0$ and $E$ preserve products, and thus we can enrich, tensor, and cotensor $\b{Op}^G_+$ over $\b{sSet}$ by using the hom objects
	\[
	\ul{\b{Op}}^G_+(\s{O},\s{P}) = E\b{Op}^G_+(\s{O},\s{P}),
	\]
and setting $K \otimes \s{O} = \coprod_{K_0} \s{O}$ and $\s{O}^K = \prod_{K_0} \s{O}$ for any $K \in \b{sSet}$ and $\s{O} \in \b{Op}^G_+$. Every hom space $\ul{\b{Op}}^G_+(\s{O},\s{P})$ is either empty or contractible.

Declare a morphism $\vp : \s{O} \to \s{P}$ in $\b{Op}^G_+$ to be a weak equivalence if it is a weak equivalence in $\b{Op}^G$. Next, let
	\[
	F_+ : \b{Sym}^G \leftrightarrows \b{Op}^G_+ : U
	\]
be the free-forgetful adjunction, so that $F_+(S) \cong F(\frac{G \times \Sigma_0}{G} \sqcup \frac{G \times \Sigma_2}{G} \sqcup S)$. We take
	\[
	\s{I}_+ = \Bigg\{ F_+\Bigg(\varnothing \,\,\longrightarrow\,\, \frac{G \times \Sigma_n}{\Gamma} \Bigg)
		\,\Bigg|\,
		\begin{array}{c}
			n \geq 0 , \, \Gamma \subset G \times \Sigma_n \t{ a}	\\
			\t{graph subgroup}
		\end{array}\Bigg\}
	\]
as a set of \emph{generating cofibrations}, and
	\[
	\s{J}_+ = \Bigg\{ F_+\Bigg( \frac{G \times \Sigma_n}{\Gamma} \,\,\stackrel{\t{inc}}{\longrightarrow}\,\, \frac{G \times \Sigma_n}{\Gamma} \sqcup \frac{G \times \Sigma_n}{\Gamma}\Bigg)
		\,\Bigg|\,
		\begin{array}{c}
			n \geq 0 , \, \Gamma \subset G \times \Sigma_n \t{ a}	\\
			\t{graph subgroup}
		\end{array}
		\Bigg\}
	\]
as a set of \emph{generating acyclic cofibrations}. These data determine a model category structure on $\b{Op}^G_+$ that is compatible with the simplicial enrichment, and which presents the homotopy theory of $N_\infty$ operads.

\begin{thm}[\cite{RubComb}]\label{thm:modcat} The category $\b{Op}^G_+$ is a right proper, combinatorial, simplicial model category, with weak equivalences and generating (acyclic) cofibrations as above. Moreover:
	\begin{enumerate}
		\item{}every object of $\b{Op}^G_+$ is fibrant, and
		\item{}if $\s{O} \in \b{Op}^G_+$ is cofibrant, then $\s{O}$ is a $N$ operad, and $\bb{E}\s{O}$ is a $N_\infty$ operad.
	\end{enumerate}
Let $Q$ be a cofibrant replacement functor on $\b{Op}^G_+$. Then the composite
	\[
	\bb{LE} = \bb{E} \circ Q : \b{Op}^G_+ \to N_\infty\t{-}\b{Op}^G
	\]
induces a DK equivalence between the hammock localizations of $\b{Op}^G_+$ and $N_\infty\t{-}\b{Op}^G$.
\end{thm}

Not every marked $N$ operad is cofibrant in $\b{Op}^G_+$; such operads are more closely akin to $\Sigma$-cofibrant operads in the sense of Berger and Moerdijk \cite{BergMoerd}.

The equivalences between the homotopy theories of $N\t{-}\b{Op}^G$, $\b{Op}^G_+$, and $N_\infty\t{-}\b{Op}^G$ enable us to analyze the homotopy theory of $N_\infty$ operads in purely combinatorial terms. As illustrated in appendix \ref{app:quotop}, interesting questions about $N_\infty$ operads transform into intricate word problems for operads in $G$-sets.

We end with a small observation. Combining Theorems \ref{thm:classify}, \ref{thm:Nopmod}, and \ref{thm:modcat} yields equivalences $\t{Ho}(\b{Op}^G_+) \simeq \b{Tr}(G) \simeq \t{Ho}(N\t{-}\b{Op}^G)$ that send an operad $\s{O} \in \b{Op}^G_+$ to the transfer system $\to_{\abs{EQ\s{O}}}$ and an operad $\s{O} \in N\t{-}\b{Op}^G$ to $\to_{\abs{E\s{O}}}$. However, for any graph subgroup $\Gamma \subset G \times \Sigma_n$, we have
	\[
	\abs{EQ\s{O}(n)}^\Gamma \neq \varnothing \iff \s{O}(n)^\Gamma \neq \varnothing \iff \abs{E\s{O}(n)}^\Gamma \neq \varnothing .
	\]
Thus, we extend Definitions \ref{defn:tsNinfty} and \ref{defn:Tadm}.

\begin{defn}\label{defn:tsO} Suppose $\s{O}$ is an operad in $N\t{-}\b{Op}^G$ or $\b{Op}^G_+$. Define a binary relation $\to_{\s{O}}$  on $\b{Sub}(G)$ by
	\[
	K \to_{\s{O}} H	\quad\t{if and only if}\quad		K \subset H \t{ and } \s{O}(\abs{H:K})^{\Gamma(H/K)} \neq \varnothing,
	\]
where $\Gamma(H/K)$ is the graph of a chosen permutation representation of $H/K$.

Similarly, for any subgroup $H \subset G$ and finite $H$-set $T$, we say that \emph{$\s{O}$ admits $T$} if $\s{O}(\abs{T})^{\Gamma(T)} \neq \varnothing$, and we write $A(\s{O})$ for the class of admissible sets of $\s{O}$.
\end{defn}

\begin{cor} For any operad $\s{O}$ in $N\t{-}\b{Op}^G$ or $\b{Op}^G_+$, the relation $\to_{\s{O}}$ is a transfer system, the class $A(\s{O})$ is an indexing system, and the functors
	\[
	\b{Op}^G_+ \stackrel{A}{\longrightarrow} \b{Ind}(G) \cong \b{Tr}(G) \quad\t{and}\quad N\t{-}\b{Op}^G \stackrel{A}{\longrightarrow} \b{Ind}(G) \cong \b{Tr}(G)
	\]
induce equivalences $\t{Ho}(\b{Op}^G_+) \simeq \b{Ind}(G) \cong \b{Tr}(G) \simeq \t{Ho}(N\t{-}\b{Op}^G)$ of $1$-categories.
\end{cor}

\begin{proof} The relations $\to_{\abs{EQ\s{O}}}$, $\to_{\s{O}}$, and $\to_{\abs{E\s{O}}}$ are equal, and the classes $A(\abs{EQ\s{O}})$, $A(\s{O})$, and $A(\abs{E\s{O}})$ are also equal.
\end{proof}

\section{Algebraic meets and joins}\label{sec:meetjoin}

In this brief section, we record how to compute meets and joins of transfer systems, and then we give a few examples. One could also work on the level of indexing systems, but this makes the mathematics more complicated. For any $G$-indexing systems $\c{I}$ and $\c{J}$, the meet $\c{I} \land \c{J}$ is just the intersection $\c{I} \cap \c{J}$, but the join $\c{I} \lor \c{J}$ is the indexing system generated the union $\c{I} \cup \c{J}$. It can be obtained by closing up $\c{I} \cup \c{J}$ under coproducts and self-induction, but this description is somewhat inexplicit.

In contrast, there is a simple formula for the join of $G$-transfer systems. It says that the join $\to \lor \rightsquigarrow$ is obtained by composing the transfers in $\to$ and $\rightsquigarrow$.

\begin{prop}\label{prop:lattr} Suppose that $G$ is a finite group, and that $\to$ and $\rightsquigarrow$ are $G$-transfer systems. Then:
	\begin{enumerate}
		\item{}the meet $\to \land \rightsquigarrow$ is the intersection $\to \cap \rightsquigarrow$, and
		\item{}the join $\to \lor \rightsquigarrow$ is the transitive closure of $\to \cup \rightsquigarrow$.
	\end{enumerate}
\end{prop}

\begin{proof} For (1), note that an intersection of transfer systems is still a transfer system. The same is not true for unions, and therefore $\to \lor \rightsquigarrow$ is the least transfer system that contains the union of $\to$ and $\rightsquigarrow$. Denote it $\la \to \cup \rightsquigarrow \ra$. By \cite[Theorem B.2]{RubChar}, the relation $\la \to \cup \rightsquigarrow \ra$ can be obtained by closing up $\to \cup \rightsquigarrow$ under conjugation and restriction, and then passing to the reflexive and transitive closure. However, the relation $\to \cup \rightsquigarrow$ is already closed under conjugation and restriction, and it is already reflexive.
\end{proof}

We illustrate these operations below.

\begin{ex} Suppose first that $G = C_{p^3}$ for a prime $p$. The subgroup lattice of $G$ is the tower $C_1 \subset C_p \subset C_{p^2} \subset C_{p^3}$, and the lattice $\b{Tr}(C_{p^3})$ of all $C_{p^3}$-transfer systems is isomorphic to the associahedron $K_5$ (cf. \cite{BBR}). Here are a few meets and joins in $\b{Tr}(C_{p^3})$.
	\begin{align*}
		\cpppb \land \cppph = \cpppa	\quad&\quad	\cpppc \land \cpppl = \cpppb \quad\quad \cpppg \land \cpppm = \cpppf	\\
		\cpppb \lor \cppph = \cpppi	\quad&\quad	\cpppc \lor \cpppl = \cpppm \quad\quad \cpppg \lor \cpppm = \cpppn
	\end{align*}

Next, suppose that $G = K_4$ is the Klein four group. Then $G$ has three proper, nontrivial subgroups of order 2, which are pairwise incomparable. The lattice $\b{Tr}(K_4)$ consists of a pair of stacked $3$-cubes, plus a layer of three vertices connecting them (cf. \cite{RubChar}). Here are a few meets and joins in $\b{Tr}(K_4)$.
	\begin{align*}
		\kindabc \land \kindkc = \kindab	\quad&\quad	\kindkct \land \kindkat = \kindkt	\quad\quad	\kindb \land \kindkb = \kindt	\\
		\kindabc \lor \kindkc = \kindkct		\quad&\quad	\kindkct \lor \kindkat = \kindkac	\quad\quad	\kindb \lor \kindkb = \kindkbt
	\end{align*}
\end{ex}

We have complete knowledge of the lattice $\b{Tr}(G)$ when $G = C_{p^3}$ or $K_4$, and therefore these meets and joins may be determined by inspection. In general, the lattice $\b{Tr}(G)$ can be quite intricate, but the formulas in Proposition \ref{prop:lattr} work regardless.

\section{Operadic products, coproducts, and tensor products}\label{sec:prodcopten}

We now relate meets and joins of indexing systems to products, coproducts, and Boardman-Vogt tensor products of operads in $\b{Set}^G$. The case for products is straightforward and was analyzed in \cite{BH}, but the cases for coproducts and tensor products are less so. We begin by recalling the correspondence between products of $N$ operads and meets of indexing systems (Proposition \ref{prop:Nopprod}), and then we show that coproducts and tensor products of $N$ operads correspond to joins of indexing systems, under suitable cofibrancy conditions (Theorems \ref{thm:Ncoprjoin} and \ref{thm:NopBV}). We briefly describe the situation for marked operads in \S\ref{subsec:markoplat}, and then in \S\ref{subsec:topprodcopten}, we discuss how these discrete constructions translate over to topology.

\subsection{Constructions on $N$ operads} We start by lifting meets and joins to the level of $N$ operads. It is more natural to work with indexing systems instead of transfer systems in this context, but the identity $\to_{A(\s{O})} \,\, = \,\, \to_{\s{O}}$ allows us to convert between the two formalisms. 

\begin{prop}[\protect{\cite[Proposition 5.1]{BH}}]\label{prop:Nopprod} If $\s{O}$ and $\s{P}$ are $N$ operads, then their product $\s{O} \times \s{P}$ in $\b{Op}^G$ is a $N$ operad, and $A(\s{O} \times \s{P}) \,\, = \,\, A(\s{O}) \land A(\s{P})$.
\end{prop}

\begin{proof} Products in $\b{Op}^G$ are computed levelwise, and therefore
	\[
	(\s{O} \times \s{P})(n)^\Xi \cong \s{O}(n)^\Xi \times \s{P}(n)^\Xi
	\]
for every $n \geq 0$ and subgroup $\Xi \subset G \times \Sigma_n$. The left hand side is nonempty if and only if both factors on the right hand side are, and the result follows.
\end{proof}

\begin{cor} For any $\s{O} \in N\t{-}\b{Op}^G$, the functor $\s{O} \times (-) : N\t{-}\b{Op}^G \to N\t{-}\b{Op}^G$ preserves weak equivalences.
\end{cor}

We have the following consistency check.

\begin{ex}\label{ex:coindlat} One standard construction of $N$ operads proceeds by coinducing the associativity operad $\b{As}$ in $\b{Set}$ up to a $G$-operad. Explicitly, if $X$ is a nonempty, right $G$-set, then $\b{Set}(X,\b{As})$ is a $N$ operad in $\b{Set}^G$. The operad $\b{Set}(X,\b{As})$ admits a finite $H$-set $T$ if and only if every $h \in H$ that fixes an element of $X$ acts as the identity on all of $T$ \cite[Proposition 3.7]{RubComb}. In particular, the $N$ operad $\b{Set}(G,\b{As})$ admits all finite $H$-sets for all subgroups $H \subset G$. It is isomorphic to the object operad of the $G$-Barratt-Eccles operad $\s{P}_G$ (cf. \cite{GM}).

For any nonempty, right $G$-sets $X$ and $Y$, there is an isomorphism
	\[
	\b{Set}(X,\b{As}) \times \b{Set}(Y,\b{As}) \cong \b{Set}(X \sqcup Y,\b{As}),
	\]
and the equality $A(\b{Set}(X \sqcup Y,\b{As})) = A(\b{Set}(X,\b{As})) \land A(\b{Set}(Y,\b{As}))$ follows from the admissibility criterion above and the fact that $g \in G$ fixes an element of $X \sqcup Y$ if and only if it fixes an element of $X$ or it fixes an element of $Y$.
\end{ex}

A dual result relates operadic coproducts to joins of indexing systems, but it is harder. Our proof relies on a presentation of the coproduct operad, which is analogous to the usual presentation for the coproduct of nonabelian groups. We refer the reader to \cite[\S\S6--8]{RubComb} for further discussion of free and quotient operads.

\begin{thm}\label{thm:Ncoprjoin} If $\s{O}$ and $\s{P}$ are $N$ operads, then their coproduct $\s{O} * \s{P}$ in $\b{Op}^G$ is also a $N$ operad, and $A( \s{O} * \s{P}) \,\, = \,\, A(\s{O}) \lor A(\s{P})$.
\end{thm}

We single out a special case before going into the proof.

\begin{ex}\label{ex:admcoprfree} Suppose $S$ is a $\Sigma$-free symmetric sequence in $\b{Set}^G$ such that $S(n)^G \neq \varnothing$ for $n = 0,2$. Now let $\s{O} = F(S)$ be the free operad on $S$. By \cite[Theorem 5.6]{RubComb}, the class $A(\s{O})$ is the indexing system generated by $A(S)$.

Now suppose $T$ is another such a symmetric sequence, and let $\s{P} = F(T)$. Then $\s{O} * \s{P} \cong F(S \sqcup T)$, and therefore
	\[
	A(\s{O} * \s{P}) = \la A(S \sqcup T) \ra = \la A(S) \cup A(T) \ra = \la A(S) \ra \lor \la A(T) \ra = A(\s{O}) \lor A(\s{P}).
	\]
Therefore Theorem \ref{thm:Ncoprjoin} is true in this case.
\end{ex}

To prove the general case, we reduce to the calculation for frees.

\begin{proof}[Proof of Theorem \ref{thm:Ncoprjoin}] Let $F : \b{Sym}^G \leftrightarrows \b{Op}^G : U$ be the free-forgetful adjunction. In Example-Lemma \ref{exlem:Nopcop}, we prove that $\s{O} * \s{P}$ is isomorphic to a sub-symmetric sequence of $F(U\s{O} \sqcup U\s{P})$ equipped with a modified operadic composition. It follows that $\s{O} * \s{P}$ is $\Sigma$-free, and since there is an operad map $\s{O} \to \s{O} * \s{P}$, it also follows that $(\s{O} * \s{P})(n)^G \neq \varnothing$ for all $n \geq 0$. Therefore $\s{O} * \s{P}$ is a $N$ operad.

The structure maps $\s{O} \to \s{O} * \s{P} \leftarrow \s{P}$ imply that $A(\s{O}) \subset A(\s{O} * \s{P}) \supset A(\s{P})$, and therefore $A(\s{O}) \lor A(\s{P}) \subset A(\s{O} * \s{P})$. On the other hand, the symmetric sequence map $U(\s{O} * \s{P}) \to UF(U\s{O} \sqcup U\s{P})$ implies $A(\s{O} * \s{P}) \subset A(F(U\s{O} \sqcup U\s{P}))$. As in Example \ref{ex:admcoprfree}, we have
	\[
	A(F(U\s{O} \sqcup U\s{P})) = \la A(U\s{O} \sqcup U\s{P}) \ra = \la A(\s{O}) \ra \lor \la A(\s{P}) \ra ,
	\]
and this equals $A(\s{O}) \lor A(\s{P})$ because $A(\s{O})$ and $A(\s{P})$ are already indexing systems. This proves the theorem.
\end{proof}

\begin{cor} For any $\s{O} \in N\t{-}\b{Op}^G$, the functor $\s{O} * (-) : N\t{-}\b{Op}^G \to N\t{-}\b{Op}^G$ preserves weak equivalences.
\end{cor}

We now consider Boardman-Vogt tensor products $\s{O} \otimes \s{P}$ of operads. Recall that $\s{O} \otimes \s{P}$ is the quotient of the coproduct $\s{O} * \s{P}$ operad by the vertical-horizontal interchange relations
	\[
	h(f(x_{11},\dots,x_{1n}),\dots,f(x_{m1},\dots,x_{mn})) \sim f(h(x_{11},\dots,x_{m1}),\dots,h(x_{1n},\dots,x_{mn})) ,
	\]
for all $h \in \s{O}(m)$ and $f \in \s{P}(n)$. When $m = n = 2$, we recover the usual formula from the Eckmann-Hilton argument. More formally, we start with the coproduct $i : \s{O} \to \s{O} * \s{P} \leftarrow \s{P} : j$, and then take the quotient by the congruence relation generated by
	\[
	\gamma(i(h) ; j(f) , \dots , j(f) ) \sim \gamma( j(f) ; i(h) , \dots , i(h) ) \sigma	,
	\]
where $h \in \s{O}(m)$, $f \in \s{P}(n)$, and $\sigma$ is the permutation that reorders $mn$ elements in reverse lexicographic order. Nullary interchanges are allowed. If $f \in \s{P}(0)$, then $\gamma(i(h) ; j(f) ,\dots , j(f)) \sim j(f)$, and therefore $i(h) \sim j(f)$ if $h \in \s{O}(0)$ as well. It follows that the operad $\s{O} \otimes \s{P}$ is reduced if both $\s{O}(0)$ and $\s{P}(0)$ are nonempty.

The tensor product of $N$ operads is not generally a $N$ operad. For example, if $G$ is the trivial group and $\s{O} = \s{P} = \b{As}$ is the associativity operad, then $\s{O} \otimes \s{P}$ is isomorphic to the commutativity operad \cite[Proposition 3.8]{FiedVogt}. However, the tensor product does behave well for suitably free $N$ operads. We introduce some terminology.

\begin{defn}A $N$ operad $\s{O}$ is \emph{cofibrant} if it is a retract of free operad in $\b{Op}^G$.
\end{defn}

This terminology is justified because there is a model category structure on $\b{Op}^G$ for which these are the cofibrant operads (cf. \cite[\S4.1]{RubComb}).

If a cofibrant $N$ operad $\s{O}$ is a retract of a free operad $F(S)$, then $S$ must be $\Sigma$-free because the composite $S \to F(S) \to \s{O}$ of the unit and the retraction is a map of symmetric sequences. After enlarging $S$, we may also assume that $S(n)^G \neq \varnothing$ for all $n \geq 0$, because $\s{O}(n)^G \neq \varnothing$ implies $F(S)(n)^G \neq \varnothing$, and therefore the inclusion $F(S) \hookrightarrow F(S \sqcup \coprod_{n \geq 0} \frac{G \times \Sigma_n}{G})$ has a retraction.

\begin{thm}\label{thm:NopBV} If $\s{O}$ and $\s{P}$ are cofibrant $N$ operads, then their tensor product $\s{O} \otimes \s{P}$ in $\b{Op}^G$ is a $N$ operad, and $A(\s{O} \otimes \s{P}) = A(\s{O}) \lor A(\s{P})$.
\end{thm}

\begin{proof} Admissible sets are preserved under retracts, so it is enough to prove the result when $\s{O} = F(S)$ and $\s{P} = F(T)$ are free on $\Sigma$-free symmetric sequences $S$ and $T$ such that $S(n)^G,T(n)^G \neq \varnothing$ for all $n \geq 0$. In this case, Lemma \ref{lem:freeopten} implies that $F(S) \otimes F(T)$ is isomorphic to a sub-symmetric sequence of $F(S \sqcup T)$, equipped with a modified composition operation. From here, the same argument used in the proof of Theorem \ref{thm:Ncoprjoin} shows that $F(S) \otimes F(T)$ is a $N$ operad and that $A(F(S) \otimes F(T)) = A(F(S)) \lor A(F(T))$.
\end{proof}

\begin{cor}If $\s{O}$ is a cofibrant $N$ operad, then $\s{O} \otimes (-) : N\t{-}\b{Op}^G \to N\t{-}\b{Op}^G$ preserves weak equivalences between cofibrant $N$ operads.
\end{cor}

\subsection{Constructions on marked operads}\label{subsec:markoplat} Next, we briefly indicate how the structure on the lattice $\b{Ind}(G)$ of $G$-indexing systems is reflected on the level of marked operads. Recall that $\b{F}$ is the free operad in $\b{Op}^G$ on the symmetric sequence $\frac{G \times \Sigma_0}{G} \sqcup \frac{G \times \Sigma_2}{G}$, and $\b{Op}^G_+ = \b{F}/\b{Op}^G$.

The relationship between products and meets is the same as before, because limits in $\b{Op}^G_+$ are computed in $\b{Op}^G$.

\begin{prop}If $\s{O}$ and $\s{P}$ are operads in $\b{Op}^G_+$, then
	\[
	A(\s{O} \times \s{P}) \,\, = \,\, A(\s{O}) \land A(\s{P}).
	\]
Consequently, the functor $\s{O} \times (-) : \b{Op}^G_+ \to \b{Op}^G_+$ preserves weak equivalences for any operad $\s{O} \in \b{Op}^G_+$.
\end{prop}

A dual result holds for derived coproducts in $\b{Op}^G_+$, with respect to the model structure discussed in \S\ref{subsec:combops}.

\begin{prop}\label{prop:OpG+cop} Suppose $\s{O}$ and $\s{P}$ are cofibrant operads in $\b{Op}^G_+$, and let $\s{O} *_\b{F} \s{P}$ be their coproduct in $\b{Op}^G_+$. Then $\s{O} *_{\b{F}} \s{P}$ is also cofibrant, and
	\[
	A(\s{O} *_{\b{F}} \s{P}) \,\, = \,\, A(\s{O}) \lor A(\s{P}).
	\]
Consequently, the functor $\s{O} *_{\b{F}} (-) : \b{Op}^G_+ \to \b{Op}^G_+$ preserves weak equivalences between cofibrant operads whenever $\s{O}$ is cofibrant in $\b{Op}^G_+$.
\end{prop}

\begin{proof} The operad $\s{O} *_{\b{F}} \s{P}$ is cofibrant for formal reasons. To compute its indexing system, note that the maps $F_+(\varnothing \to \frac{G \times \Sigma_n}{\Gamma})$ are generating cofibrations for $\b{Op}^G_+$, and therefore every cofibrant operad is a retract of a free operad $F_+(S)$ for some $\Sigma$-free symmetric sequence $S$. Indexing systems are preserved under retracts, and therefore it will suffice to prove that $A(\s{O} *_{\b{F}} \s{P}) \,\, = \,\, A(\s{O}) \lor A(\s{P})$ when $\s{O}$ and $\s{P}$ are free in this sense.

Suppose that $\s{O} = F_+(S)$ and $\s{P} = F_+(T)$ for $\Sigma$-free symmetric sequences $S$ and $T$. Then $\s{O} *_{\b{F}} \s{P} \cong F_+(S \sqcup T)$, and the identity $A(\s{O} *_{\b{F}} \s{P}) = A(\s{O}) \lor A(\s{P})$ follows as in Example \ref{ex:admcoprfree}.
\end{proof}

One can also construct Boardman-Vogt tensor products $\otimes_{\b{F}}$ in $\b{Op}^G_+$, but they are quite pathological. Morphisms in $\b{Op}^G_+$ must preserve markings, and therefore the distinguished binary operation $p$ in the tensor product $\s{O} \otimes_{\b{F}} \s{P}$ must interchange with itself, i.e. $p(p(x,y),p(z,w)) = p(p(x,z),p(y,w))$. Thus, the cycle $(23) \in \Sigma_4$ stabilizes $q = \gamma(p;p,p)$, and it follows that the operad $\s{O} \otimes_{\b{F}} \s{P}$ is never $\Sigma$-free or cofibrant. Moreover, if $g \in G$ is an element of order $2$ and $\Gamma = \{ (1,\t{id}) , (g,(23)) \}$, then $q$ is $\Gamma$-fixed, which makes it a $\la g \ra/1$-norm. In this case, taking $\s{O} = \s{P} = \b{F}$ yields the inequality $A(\s{O} \otimes_{\b{F}} \s{P}) \supsetneq A(\s{O}) \lor A(\s{P})$. For these reasons, we pursue $\otimes_{\b{F}}$ no further.

\subsection{Topological interpretations}\label{subsec:topprodcopten} We conclude this section by considering how constructions on the level of discrete $G$-operads translate into constructions for topological $N_\infty$ operads. As observed earlier, the functor $\bb{E} = \abs{\cdot} \circ E$ preserves products of operads because it preserves finite limits. Unfortunately, it does not preserve point-set level operadic coproducts or tensor products, and therefore our constructions in $N\t{-}\b{Op}^G$ and $\b{Op}^G_+$ need to be interpreted carefully.

We think of the derived coproduct and tensor product in $N\t{-}\b{Op}^G$ and $\b{Op}^G_+$ as the homotopically correct constructions, and we read off their topological counterparts via the equivalences $\t{Ho}(N\t{-}\b{Op}^G) \simeq \t{Ho}(N_\infty\t{-}\b{Op}^G) \simeq \t{Ho}(\b{Op}^G_+)$. For simplicity, we focus on the unmarked case. Let $(-)^u : N_\infty\t{-}\b{Op}^G \to N\t{-}\b{Op}^G$ be the forgetful functor that ignores all topology. This functor preserves admissible sets, and therefore it also preserves weak equivalences. It is a homotopical inverse to $\bb{E}$.

\begin{prop}\label{prop:Eu}The homotopical functors
	\[
	\bb{E} : N\t{-}\b{Op}^G \leftrightarrows N_\infty\t{-}\b{Op}^G : (-)^u
	\]
are inverse up to zig-zags of natural weak equivalences.
\end{prop}

\begin{proof} We use the product trick from \cite{MayGILS}. Both $\bb{E}$ and $(-)^u$ preserve admissible sets, and therefore both projections in the diagram $\bb{E}(\s{O}^u) \leftarrow \bb{E}(\s{O}^u) \times \s{O} \to \s{O}$ are weak equivalences that are natural in the operad $\s{O} \in N_\infty\t{-}\b{Op}^G$. Similarly for the other composite.
\end{proof}

Accordingly, we define $N_\infty$ coproducts and tensor products by ignoring topology, performing the combinatorial constructions, and then inserting cells.

\begin{defn} For any $N_\infty$ operads $\s{O}$ and $\s{P}$, define
	\[
	\s{O} *^{N_\infty} \s{P} = \bb{E}(\s{O}^u * \s{P}^u)	\quad\t{and}\quad	\s{O} \otimes^{N_\infty} \s{P} = \bb{E}(F\s{O}^u \otimes F\s{P}^u),
	\]
where $F$ is the free operad functor $F : \b{Sym}^G \to \b{Op}^G$.
\end{defn}

We have $A(\s{O} *^{N_\infty} \s{P})  = A(\s{O}) \lor A(\s{P}) = A(\s{O} \otimes^{N_\infty} \s{P})$ by Theorems \ref{thm:Ncoprjoin} and \ref{thm:NopBV}, so these constructions have the correct behavior. One can also construct an analogous $N_\infty$ product, but the equivalences 
	\[
	\s{O} \times^{N_\infty} \s{P} = \bb{E}(\s{O}^u \times \s{P}^u) \cong (\bb{E}\s{O}^u) \times (\bb{E}\s{P}^u) \simeq \s{O} \times \s{P},
	\]
show that it is unnecessary.

The cells attached by $\bb{E}$ in $*^{N_\infty}$ and $\otimes^{N_\infty}$ spoil the point-set level universal properties of $*$ and $\otimes$. For this reason, it is natural to ask when a $N_\infty$ coproduct or tensor product is equivalent to the usual operadic coproduct or tensor product. We do not believe that $N_\infty$ coproducts are ever equivalent to ordinary ones. As the following example illustrates, it is quite difficult construct an action by $\s{O} *^{N_\infty} \s{P}$ from separate actions of $\s{O}$ and $\s{P}$.

\begin{ex}\label{ex:Ninftycop} Suppose that $X$ is a $G$-space, and that $\s{O}$ and $\s{P}$ are two $N_\infty$ $G$-operads. If $\s{O} * \s{P}$ and $\s{O} *^{N_\infty} \s{P}$ are equivalent, then actions by these operads should consist of equivalent data. An action of $\s{O} * \s{P}$ on $X$ is the same thing as an action by $\s{O}$ and an action by $\s{P}$. On the other hand, if we have an action by $\s{O} *^{N_\infty} \s{P}$, then for any $n \geq 0$ and  $f \in \s{O}(n)$ and $h \in \s{P}(n)$, we must have coherence homotopies between the corresponding operations $F$ and $H$ on $X$. This is not obviously part of the given data, because $\s{O}$ only parametrizes coherence homotopies between its operations, and similarly for $\s{P}$.

There is a bit more hope if we work in a marked setting. Suppose $\s{O}$ and $\s{P}$ have distinguished units $u_{\s{O}} \in \s{O}(0)^G$ and $u_{\s{P}} \in \s{P}(0)^G$ and distinguished binary products $p_{\s{O}} \in \s{O}(2)^G$ and $p_{\s{P}} \in \s{P}(2)^G$, which represent the same operations $U$ and $P$ on $X$. Then $\s{O}$ parametrizes a homotopy from $F$ to $P(\dots P(P(x_1,x_2),x_3),\dots,x_n)$ and $\s{P}$ parametrizes a homotopy from $P(\dots P(P(x_1,x_2),x_3),\dots,x_n)$ to $H$. The issue now is that there should be a $\t{Stab}_{G \times \Sigma_n}(f) \cap \t{Stab}_{G \times \Sigma_n}(h)$-fixed homotopy from $F$ to $H$, and the homotopy above does not necessarily have this property.
\end{ex}

The situation is less clear cut for tensor products. The Dunn additivity theorem (cf. \cite{Dunn}, \cite{FiedVogt}) asserts that the tensor product of an $E_k$-operad with an $E_l$-operad is $E_{k+l}$, provided that the operads are suitably cofibrant. This motivated \cite[Conjecture 6.27]{BH}, which we reproduce below.

\begin{conj}\label{conj:BH} If $\s{O}$ and $\s{P}$ are suitably cofibrant $N_\infty$ operads, then $\s{O} \otimes \s{P}$ is also a $N_\infty$ operad, and $A(\s{O} \otimes \s{P}) = A(\s{O}) \lor A(\s{P})$.
\end{conj}

For such operads, we would have $\s{O} \otimes^{N_\infty} \s{P} \simeq \s{O} \otimes \s{P}$ because both sides would have the same admissible sets. Theorem \ref{thm:NopBV} is a precise combinatorial analogue to this conjecture, but it does not quite imply the topological result. We can deduce the following portions though.

\begin{prop}\label{prop:combtentop} If $\s{O}$ and $\s{P}$ are cofibrant $N$ operads, then $\bb{E}\s{O} \otimes \bb{E}\s{P}$ is $\Sigma$-free, has $G$-fixed points of all arities, and satisfies $A(\bb{E}\s{O} \otimes \bb{E}\s{P}) = A(\bb{E}\s{O}) \lor A(\bb{E}\s{P})$.
\end{prop}

\begin{proof} By Theorem \ref{thm:NopBV}, the tensor product $\s{O} \otimes \s{P}$ in $\b{Op}^G$ is a $N$ operad such that $A(\s{O} \otimes \s{P}) = A(\s{O}) \lor A(\s{P})$. Applying the functor $\bb{E}$ gives a pair of interchanging maps $\bb{E}\s{O} \to \bb{E}(\s{O} \otimes \s{P}) \leftarrow \bb{E}\s{P}$ of $N_\infty$ operads, which in turn induce a map $\bb{E}\s{O} \otimes \bb{E}\s{P} \to \bb{E}(\s{O} \otimes \s{P})$ by universality. Since $\bb{E}(\s{O} \otimes \s{P})$ is $\Sigma$-free, so is $\bb{E}\s{O} \otimes \bb{E}\s{P}$. Since $\bb{E}\s{O}(n)^G \neq \varnothing$ for all $n \geq 0$, the same holds for $\bb{E}\s{O} \otimes \bb{E}\s{P}$. The existence of a map $\bb{E}\s{O} \otimes \bb{E}\s{P} \to \bb{E}(\s{O} \otimes \s{P})$ implies the inclusion
	\[
	A(\bb{E}\s{O} \otimes \bb{E}\s{P}) \subset A(\bb{E}(\s{O} \otimes \s{P})) = A(\s{O}) \lor A(\s{P}) = A(\bb{E}\s{O}) \lor A(\bb{E}\s{P}),
	\]
and the reverse inclusion follows from the maps $\bb{E}\s{O} \to \bb{E}\s{O} \otimes \bb{E}\s{P} \leftarrow \bb{E}\s{P}$ in the universal diagram.
\end{proof}

Thus, if one could prove that $(\bb{E}\s{O} \otimes \bb{E}\s{P})(n)^\Gamma$ is either empty or contractible for every integer $n \geq 0$ and graph subgroup $\Gamma \subset G \times \Sigma_n$, then Conjecture \ref{conj:BH} would hold for the $N_\infty$ operads $\bb{E}\s{O}$ and $\bb{E}\s{P}$. That being said, we can already deduce useful topological results without knowing that $\bb{E}\s{O} \otimes \bb{E}\s{P}$ is a $N_\infty$ operad.

\begin{thm}\label{thm:Ninftyinterchange} Let $R$ be an $\s{O}$-algebra orthogonal $G$-spectrum for some $N_\infty$ operad $\s{O}$, and suppose further that $\t{id} \in \s{O}(1)^G$ is a nondegenerate basepoint, and $\s{O}(n)$ is of the homotopy type of a $G \times \Sigma_n$-CW complex for every $n \geq 0$. Then there is a weakly equivalent $G$-spectrum $R' \simeq R$ and a weakly equivalent $N_\infty$ operad $\s{O}' \simeq \s{O}$ such that $R'$ is equipped with a pair of interchanging $\s{O}'$-actions.
\end{thm}

\begin{proof}Let $\s{O}' = \bb{E}(F\s{O}^u)$ and $\s{P} = \bb{E}(F\s{O}^u \otimes F\s{O}^u)$. Then $\s{O}$, $\s{O}'$ and $\s{P}$ are equivalent $N_\infty$ operads. The projections $\s{O} \leftarrow \s{O} \times \s{P} \to \s{P}$ induce a chain of Quillen equivalences $\b{Sp}^G[\s{O}] \simeq \b{Sp}^G[\s{O} \otimes \s{P}] \simeq \b{Sp}^G[\s{P}]$ between the corresponding categories of algebra $G$-spectra by \cite[Theorem A.3]{BH}, and we let $R'$ be a fibrant replacement of the image of $R$ in $\b{Sp}^G[\s{P}]$. Applying $\bb{E}$ to the universal diagram $F\s{O}^u \to F\s{O}^u \otimes F\s{O}^u \leftarrow F\s{O}^u$ gives a pair of interchanging maps $\s{O}' \to \s{P} \leftarrow \s{O}'$, and pulling back gives a pair of interchanging $\s{O}'$-algebra structures on $R'$.
\end{proof}

If $R$ is an $\s{O}$-algebra $G$-spectrum such that the $\s{O}$-action self-interchanges, then certain norm maps in $\ul{\pi}_0(R)$ become homomorphisms of multiplicative monoids \cite[Theorem 7.12]{BH}. This is a necessary condition for $\ul{\pi}_0(R)$ to be an incomplete Tambara functor in the sense of \cite{BHinc}.

\section{Algebraic images and inverse images}\label{sec:iminvim}

In this section, we give purely algebraic definitions of image and inverse image transfer systems (Definition \ref{defn:iminvimts}), and then we establish their functoriality and adjointness properties (Proposition \ref{prop:iminvimfunts}). We relate these constructions to operadic induction, restriction, and coinduction in \S\ref{sec:indrescoind}. Much of this theory works as expected, but there are a few surprises. Most notably, there is an extra inverse image construction. Every group homomorphism $f : G \to G'$ determines a pair of image constructions analogous to induction and coinduction, but the map $f$ also determines two inverse image constructions, which happen to coincide if $f$ is injective (Proposition \ref{prop:injhomots}). When $f$ is noninjective, one of these inverse images corresponds to restriction, but the other one and its adjoint seem to be red herrings, with no natural operadic interpretation.

\subsection{Overview} We sketch the definitions and offer a few examples now, before giving a more formal treatment in the next section.

Our constructions are loosely inspired by a pair of adjunctions associated to an arbitrary set map. Suppose $X$ and $Y$ are sets, $f : X \to Y$ is a function, and $\c{P}(X)$ and $\c{P}(Y)$ are the corresponding power sets, regarded as posets under inclusion. Then taking images and inverse images determines an order adjunction $f : \c{P}(X) \leftrightarrows \c{P}(Y) : f^{-1}$. Intersections are not always preserved under images, and therefore $f : \c{P}(X) \to \c{P}(Y)$ does not always have a left adjoint. However, there is an adjunction $f^{-1} : \c{P}(Y) \leftrightarrows \c{P}(X) : f_*$, where $f_*A = \{ y \in Y \, | \, f^{-1}(y) \subset A\}$ for any subset $A \subset X$. For comparison, the ordinary image can be expressed as $fA = \{ y \in Y \, | \, f^{-1}(y) \cap A \neq \varnothing\}$. The chain of adjunctions generally stops here because the right adjoint $f_*$ does not always preserve unions.

If $f : G \to G'$ is a group homomorphism, then it makes sense to apply $f^{\times 2}$ and $(f^{-1})^{\times 2}$ to the relations in a transfer system. The results need not be transfer systems, but we can close them up. Every binary relation $R$ on $\b{Sub}(G)$ generates a transfer system $\la R \ra$, provided it refines inclusion. Explicitly, the relation $\la R \ra$ is obtained by closing $R$ under conjugation and restriction, and then passing to the reflexive and transitive closure \cite[Theorem B.2]{RubChar}. Combining $\la \cdot \ra$ with the set-theoretic maps $f^{\times 2}$ and $(f^{-1})^{\times 2}$ gives natural transfer system analogues of $f$ and $f^{-1}$. We denote them $f_L$ and $f^{-1}_L$.

\begin{ex}\label{ex:leftiminv} Consider the map $f : C_4 \to \Sigma_3$ that sends the generator of $C_4$ to a transposition. The subgroup lattice of $C_4$ is the tower $C_1 \subset C_2 \subset C_4$. The proper, nontrivial subgroups of $\Sigma_3$ are three conjugate copies of $C_2$ and a single copy of $C_3$. We draw them as three dots in a row, and an odd dot off to the side. Here are some examples of $f_L$ and $f^{-1}_L$.
	\begin{align*}
		f_L \Big( \cppa \Big) = f_L \Big( \cppb \Big) = \dsia	\quad&\quad	f_L \Big( \cppc \Big) = f_L \Big( \cppd \Big) = f_L \Big( \cppe \Big) = \dsib	\\
		f^{-1}_L \Big( \dsia \Big) = f^{-1}_L \Big( \dsic \Big) = \cppa		\quad&\quad	f^{-1}_L \Big( \dsib \Big) = f^{-1}_L \Big( \dsii \Big) = \cppd
	\end{align*}
\end{ex}

On the other hand, constructing an analogue to $f_*$ for transfer systems requires another approach, because the set-theoretic map $f_*$ does not preserve subgroups. The power set adjunction $f^{-1} \dashv f_*$ indicates that $f^{-1}$ and $f_*$ should be suitably dual, which necessitates the next construction.

\begin{prop}\label{prop:cogents} Suppose that $\leq$ is a partial order on $\b{Sub}(G)$ that refines inclusion. Then
	\[
	\ra\! \leq \!\la \,\, := \Bigg\{ (K,H) \in \b{Sub}(G)^{\times 2} \, \Bigg| \, \begin{array}{c}
	\t{$K \subset H$, and $gKg^{-1} \cap L \leq L$}\\
	\t{for all $g \in G$ and  $L \subset gHg^{-1}$}
	\end{array}
	\Bigg\}
	\]
is the largest $G$-transfer system contained in $\leq$.
\end{prop}

Note that if $R$ is any reflexive relation on $\b{Sub}(G)$, then there must be maximal transfer systems contained in $R$, but there need not be a maximum $\ra R  \la$. Assuming $R$ is a partial order allows us to construct $\ra R \la$ directly.

\begin{proof}[Proof of Proposition \ref{prop:cogents}] We begin by showing $\ra\! \leq \!\la$ is a transfer system. The reflexivity of $\ra\! \leq \!\la$ follows from that of $\leq$. By definition, the relation $\ra\! \leq \!\la$ refines $\subset$, and therefore it is also antisymmetric. For transitivity, suppose that $(K,J), (J,H) \in \, \ra\! \leq \!\la$. Given  $g \in G$ and $L \subset gHg^{-1}$, let $M = gJg^{-1} \cap L$. Then $M \subset gJg^{-1}$, and we have
	\[
	gKg^{-1} \cap L = gKg^{-1} \cap M \leq M = gJg^{-1} \cap L \leq L,
	\]
so that $gKg^{-1} \cap L \leq L$ by the transitivity of $\leq$. It is clear that $\ra \! \leq \! \la$ is closed under conjugation. For restriction, suppose $(K,H) \in \,\, \ra\! \leq \!\la$ and $L \subset H$. Then $(K \cap L , L) \in \,\, \ra\! \leq \!\la$ because if $g \in G$ and $M \subset gLg^{-1}$, then $M \subset gHg^{-1}$ and hence
	\[
	g(K \cap L)g^{-1} \cap M = gKg^{-1} \cap M \leq M.
	\]
Therefore $\ra\! \leq \!\la$ is a transfer system. It refines $\leq$ because if $(K,H) \in \,\, \ra\! \leq \!\la$, then taking $g = e \in G$ and $L = H \subset eHe^{-1}$ shows $K = eKe^{-1} \cap H \leq H$.

Finally, suppose that $\to$ is a transfer system that refines $\leq$, and suppose $K \to H$. Then $K \subset H$ because $\to$ refines inclusion. Then, for any $g \in G$ and $L \subset gHg^{-1}$, we have $gKg^{-1} \to gHg^{-1}$ and $gKg^{-1} \cap L \to L$, which implies $gKg^{-1} \cap L \leq L$. Therefore $\to$ refines $\ra\! \leq \!\la$.
\end{proof}

We obtain the transfer system analogue of $f_*$ by dualizing the construction of $f^{-1}_L$. First, we take the inverse image along $(f^{-1})^{\times 2}$, and then we apply $\ra \! \cdot \! \la$. Similarly, one can take the inverse image along $f^{\times 2}$ and then apply $\ra \! \cdot \! \la$, and this is where the extra inverse image map comes from. We denote these two constructions $f_R$ and $f^{-1}_R$.

\begin{ex}\label{ex:rightiminv} Consider the homomorphism $f : C_4 \to \Sigma_3$ from Example \ref{ex:leftiminv} once more. Here are some examples of $f_R$ and $f^{-1}_R$.
	\begin{align*}
		f_R \Big( \cppa \Big) = f_R \Big( \cppb \Big) = f_R \Big( \cppc \Big) = \dsic	\quad&\quad	f_R \Big( \cppd \Big) = f_R \Big( \cppe \Big) = \dsii	\\
		f^{-1}_R \Big( \dsia \Big) = f^{-1}_R \Big( \dsic \Big) = \cppb	\quad&\quad	f^{-1}_R \Big( \dsib \Big) = f^{-1}_R \Big( \dsii \Big) = \cppe
	\end{align*}
\end{ex}

Note the differences between Examples \ref{ex:leftiminv} and \ref{ex:rightiminv}. The maps $f_L$ and $f_R$ should be distinct, because operadic induction and coinduction are distinct. The maps $f^{-1}_L$ and $f^{-1}_R$ are both supposed to model restriction, but we are seeing a pointwise inequality $f^{-1}_L < f^{-1}_R$. This occurs for every noninjective map $f$, and the operadically correct construction turns out to be $f^{-1}_L$.

\subsection{Definitions and first properties of image and inverse image transfer systems} For any group $G$, let
	\[
	\b{Sub}(G)_{\subset} = \{ (K,H) \in \b{Sub}(G)^{\times 2} \, | \, K \subset H \}
	\]
and suppose $F : \b{Sub}(G) \to \b{Sub}(G')$ is an order-preserving map. Then $F^{\times 2}$ restricts to a set map
	\[
	F_{\subset} : \b{Sub}(G)_{\subset} \to \b{Sub}(G')_{\subset} ,
	\]
and therefore there is an image-inverse image adjunction
	\[
	F_{\subset} : \c{P}(\b{Sub}(G)_{\subset}) \rightleftarrows \c{P}(\b{Sub}(G')_{\subset}) : (F_{\subset})^{-1}.
	\]
We can identify the elements of $\c{P}(\b{Sub}(G)_{\subset})$ with binary relations on $\b{Sub}(G)$ that refine inclusion, and similarly for $G'$. Moreover, if $R \in \c{P}(\b{Sub}(G')_{\subset})$ is a partial order, then so is $(F_{\subset})^{-1}R$. This enables us to make the following definitions.

\begin{defn}\label{defn:tsFiminv} Suppose $G$ and $G'$ are finite groups and $F : \b{Sub}(G) \to \b{Sub}(G')$ is an order-preserving map. For any $G$-transfer system $\to$, define
	\[
	F_L(\to) := \la F_{\subset}(\to) \ra 
	\]
and for any $G'$-transfer system $\rightsquigarrow$, define
	\[
	F^{-1}_R(\rightsquigarrow) := \,\, \ra  (F_{\subset})^{-1}(\rightsquigarrow) \la .
	\]
\end{defn}

We summarize a few properties of $F_L$ and $F^{-1}_R$.

\begin{lem}\label{lem:iminvimts} Suppose $G$ and $G'$ are finite groups. 
	\begin{enumerate}
		\item{}For any inclusion-preserving map $F : \b{Sub}(G) \to \b{Sub}(G')$, the induced maps $F_L : \b{Tr}(G) \rightleftarrows \b{Tr}(G') : F_R^{-1}$ form an adjunction $F_L \dashv F_R^{-1}$. If $F_{\subset}$ preserves transfer systems, then $F_L = F_{\subset}$. If $(F_{\subset})^{-1}$ preserves transfer systems, then $F_R^{-1} = (F_{\subset})^{-1}$.
		\item{}For any pair of inclusion-preserving maps $E : \b{Sub}(G) \to \b{Sub}(G')$ and $F : \b{Sub}(G') \to \b{Sub}(G'')$, we have refinements
			\[
			(FE)_L(\to) \subset F_L E_L(\to) \quad\t{and}\quad E_R^{-1} F_R^{-1}(\rightsquigarrow) \subset (FE)_R^{-1}(\rightsquigarrow).
			\]
		Moreover, if either of the equalities $(FE)_L = F_L E_L$ or $(FE)^{-1}_R = E^{-1}_R F^{-1}_R$ hold, then both of them hold.
		\item{}If either $E_{\subset}$ or $(F_{\subset})^{-1}$ preserves transfer systems, then both $(FE)_L = F_L E_L$ and $(FE)_R^{-1} = E_R^{-1} F_R^{-1}$.
	\end{enumerate}
\end{lem}

\begin{proof}For $(1)$, the adjunction $F_L \dashv F^{-1}_R$ follows from the adjunction $F_\subset \dashv (F_\subset)^{-1}$ and the adjointness properties of $\la \cdot \ra$ and $\ra \! \cdot \! \la$. If $F_{\subset}$ preserves transfer systems, then applying $\la \cdot \ra$ does nothing to $F_{\subset}(\to)$, and similarly for $(F_{\subset})^{-1}$.

Now for $(2)$. Suppose $E : \b{Sub}(G) \to \b{Sub}(G')$ and $F : \b{Sub}(G') \to \b{Sub}(G'')$ are order-preserving. For any $G$-transfer system $\to$, we have $E_{\subset}(\to) \subset E_L(\to)$, and hence $(FE)_{\subset}(\to) \subset F_{\subset} E_L(\to) \subset F_L E_L(\to)$. It follows $(FE)_L(\to) \subset F_L E_L(\to)$. Dually, $E^{-1}_R F^{-1}_R(\rightsquigarrow) \subset (FE)_R^{-1}(\rightsquigarrow)$ for every $G''$-transfer system $\rightsquigarrow$.

Suppose further that $(FE)_L = F_L E_L$. Then by the uniqueness of adjoints, the functors $E_R^{-1} F_R^{-1}$ and $(FE)_R^{-1}$ are naturally isomorphic maps $\b{Tr}(G'') \rightrightarrows \b{Tr}(G)$, but the codomain is a poset. Therefore $E_R^{-1} F_R^{-1} = (FE)_R^{-1}$. The argument when $(FE)^{-1}_R = E^{-1}_R F^{-1}_R$ is dual.

For $(3)$, suppose that $E_\subset$ preserves transfer systems. Then
	\[
	F_L E_L (\to) = \la F_\subset \la E_\subset (\to) \ra \ra = \la F_\subset E_\subset (\to) \ra  = (FE)_L(\to)
	\]
for every $G$-transfer system $\to$. The equality $(FE)^{-1}_R = E^{-1}_R F^{-1}_R$ follows from $(2)$. The argument when $(F_\subset)^{-1}$ preserves transfer systems is dual.
\end{proof}

Specializing Definition \ref{defn:tsFiminv} to the case where $F$ is the image or inverse image map associated to a group homomorphism $f : G \to G'$ yields the corresponding image and inverse image maps for transfer systems.

\begin{defn}\label{defn:iminvimts} Let $f : G \to G'$ be a homomorphism between finite groups. Taking $F = f : \b{Sub}(G) \to \b{Sub}(G')$ in Definition \ref{defn:tsFiminv} determines an adjunction
	\[
	f_L : \b{Tr}(G) \rightleftarrows \b{Tr}(G') : f^{-1}_R
	\]
and taking $F = f^{-1} : \b{Sub}(G') \to \b{Sub}(G)$ determines another adjunction.
	\[
	f^{-1}_L := (f^{-1})_L : \b{Tr}(G') \rightleftarrows \b{Tr}(G) : (f^{-1})^{-1}_R =: f_R.
	\]
\end{defn}

The functoriality of $(-)_L$, $(-)_R$, $(-)^{-1}_L$, and $(-)^{-1}_R$ does not immediately follow from Lemma \ref{lem:iminvimts}. These constructions preserve identity morphisms, but for any pair of composable group homomorphisms
	\[
	G \stackrel{h}{\longrightarrow} G' \stackrel{k}{\longrightarrow} G'',
	\]
and transfer systems $\to \,\, \in \b{Tr}(G)$ and $\rightsquigarrow \,\, \in \b{Tr}(G'')$, Lemma \ref{lem:iminvimts} only implies that
	\begin{align*}
	(kh)_L(\to) \subset k_L h_L(\to)	&\quad , \quad	(kh)^{-1}_R(\rightsquigarrow) \supset h^{-1}_R k^{-1}_R(\rightsquigarrow)	\\
	(kh)^{-1}_L(\rightsquigarrow) \subset h^{-1}_L k^{-1}_L(\rightsquigarrow)	&\quad , \quad	(kh)_R(\to) \supset k_R h_R(\to) .
	\end{align*}
We establish the reverse inclusions by analyzing the precise constructions of $h_L(\to)$ and $k^{-1}_L(\rightsquigarrow)$. These transfer systems are slightly less complicated than general theory predicts.

\begin{lem}\label{lem:Liminv} Suppose $f : G \to G'$ is a homomorphism between finite groups.
	\begin{enumerate}
		\item{}For any $G$-transfer system $\to$, the $G'$-transfer system $f_L(\to)$ is the reflexive and transitive closure of the relation
			\[
			\bigcup_{K \to H} \Big\{ (g (fK) g^{-1} , g (fH) g^{-1} ) \, \Big| \, g \in G'  \Big\}.
			\]
		\item{}For any $G'$-transfer system $\rightsquigarrow$, the $G$-transfer system $f^{-1}_L(\rightsquigarrow)$ is the reflexive and transitive closure of the relation
			\[
			\bigcup_{K' \rightsquigarrow H'} \Big\{(f^{-1} K' \cap L , L) \, \Big| \, L \subset f^{-1} H' \Big\}.
			\]
	\end{enumerate}
\end{lem}

\begin{proof} The relation $f_L(\to)$ is obtained by closing $f_{\subset}(\to)$ under conjugation and restriction, and then taking the reflexive and transitive closure, but $f_{\subset}(\to)$ is already closed under restriction. If $(fK,fH) \in f_\subset(\to)$ for some $K \to H$, and $L' \subset fH$, then for $L = f^{-1} L' \cap H$ we have $K \cap L \to L$, and therefore $(fK \cap L' , L') = (f(K \cap L) , fL) \in f_\subset(\to)$. Claim (1) follows.

For claim (2), it is enough to show $(f^{-1})_{\subset}(\rightsquigarrow)$ is closed under conjugation. If $(f^{-1} K' , f^{-1} H') \in (f^{-1})_\subset(\rightsquigarrow)$ for some $K' \rightsquigarrow H'$, and $\alpha \in G$, then we have $f(\alpha) K' f(\alpha)^{-1} \rightsquigarrow f(\alpha) H' f(\alpha)^{-1}$, and therefore $(\alpha (f^{-1} K') \alpha^{-1} , \alpha (f^{-1} H') \alpha^{-1}) = ( f^{-1}(f(\alpha) K' f(\alpha)^{-1}) , f^{-1}( f(\alpha) H' f(\alpha)^{-1}) ) \in (f^{-1})_\subset(\rightsquigarrow)$. 
\end{proof}

These simplifications buy us just enough room to establish functoriality.

\begin{prop}\label{prop:revinciminvim} Suppose that $h : G \to G'$ and $k : G' \to G''$ are homomorphisms between finite groups, and $\to \,\, \in \b{Tr}(G)$ and $\rightsquigarrow \,\, \in \b{Tr}(G'')$ are transfer systems. Then:
	\begin{enumerate}
		\item{}$k_L h_L(\to) = (kh)_L(\to)$ and $k_R h_R(\to) = (kh)_R(\to)$, and
		\item{}$h^{-1}_L k^{-1}_L(\rightsquigarrow) = (kh)^{-1}_L(\rightsquigarrow)$ and $h^{-1}_R k^{-1}_R ( \rightsquigarrow) = (kh)^{-1}_R (\rightsquigarrow)$.
	\end{enumerate}
\end{prop}

\begin{proof} It will be enough to show that $k_L h_L(\to) \subset (kh)_L(\to)$ and $h^{-1}_L k^{-1}_L(\rightsquigarrow) \subset (kh)^{-1}_L(\rightsquigarrow)$. For the former inclusion, suppose $(K',H') \in h_L(\to)$. By Lemma \ref{lem:Liminv}, there is a sequence of subgroups $H'_0,\dots,H'_n \subset G'$ such that $K' = H_0'$, $H' = H_n'$, and if $0 \leq i < n$, then
	\[
	(H_i',H_{i+1}') = ( g_i (h K_i) g_i^{-1} , g_i (h H_i) g_i^{-1} )
	\]
for some $g_i \in G'$ and $K_i,H_i \subset G$ that satisfy $K_i \to H_i$. Apply $k : G' \to G''$ to the subgroups $H_i'$. Then, since conjugation is preserved under images, we have
	\[
	(k(H_i'),k(H_{i+1}')) = ( k(g_i) (kh K_i) k(g_i)^{-1} , k(g_i) (kh H_i) k(g_i)^{-1} ) \in (kh)_L(\to) ,
	\]
and thus $(kK', kH') \in (kh)_L(\to)$ by transitivity. This shows that $k_{\subset}(h_L(\to)) \subset (kh)_L(\to)$, and the inclusion $k_L(h_L(\to)) = \la k_{\subset}(h_L(\to)) \ra \subset (kh)_L(\to)$ follows.

 The proof of the inclusion $h^{-1}_L k^{-1}_L(\rightsquigarrow) \subset (kh)^{-1}_L(\rightsquigarrow)$ is similar. The inclusion $(h^{-1})_{\subset}(k^{-1}_L(\rightsquigarrow)) \subset (kh)^{-1}_L(\rightsquigarrow)$ holds because $h^{-1}$ preserves intersections, and then $h^{-1}_L(k^{-1}_L(\rightsquigarrow)) \subset (kh)^{-1}_L(\rightsquigarrow)$ follows as before.
 
 Thus $(kh)_L = k_L h_L$ and $(kh)^{-1}_L = h^{-1}_L k^{-1}_L$, and the analogous equations for $(-)_R$ and $(-)^{-1}_R$ hold by part (2) of Lemma \ref{lem:iminvimts}.
\end{proof}

 In summary, we obtain the following result.

\begin{prop}\label{prop:iminvimfunts} Let $\b{FinGrp}$ and $\b{FinPos}$ denote the categories of finite groups and finite posets. Then the constructions in Definition \ref{defn:iminvimts} determine functors
	\[
	(-)_L \, , \, (-)_R : \b{FinGrp} \rightrightarrows \b{FinPos}	\quad\t{and}\quad	(-)^{-1}_L \, , \, (-)^{-1}_R : \b{FinGrp}^{\t{op}} \rightrightarrows \b{FinPos}
	\]
such that for any homomorphism $f : G \to G'$ in $\b{FinGrp}$, there are order adjunctions $f_L \dashv f^{-1}_R$ and $f^{-1}_L \dashv f_R$.
\end{prop}

We now examine the relationship between $f^{-1}_L$ and $f^{-1}_R$. As illustrated in Examples \ref{ex:leftiminv} and \ref{ex:rightiminv}, the transfer systems $f^{-1}_L(\to)$ and $f^{-1}_R(\to)$ need not be equal. In fact, they can be maximally far apart.

\begin{ex}\label{ex:invimsneqts} Let $! : G \to 1$ be the unique morphism. There is only one transfer system $\to \,\, \in \b{Tr}(1)$, and it is both initial and terminal. Applying $!^{-1}_L$ yields the initial $G$-transfer system, because $!^{-1}_L$ is a left adjoint, and applying $!^{-1}_R$ yields the terminal transfer system.
\end{ex}

This sort of inequality holds in general.

\begin{prop}\label{prop:lrinvimineq} Suppose that $f : G \to G'$ is a homomorphism between finite groups. Then for any $\rightsquigarrow \,\, \in \b{Tr}(G')$, we have $f^{-1}_L(\rightsquigarrow) \subset f^{-1}_R(\rightsquigarrow)$.
\end{prop}

\begin{proof} For any $\rightsquigarrow \,\, \in \b{Tr}(G')$, we claim that $(f^{-1})_\subset(\rightsquigarrow) \subset \,\, \ra (f_\subset)^{-1}(\rightsquigarrow) \la$. For suppose $(K,H) \in (f^{-1})_\subset(\rightsquigarrow)$. Then $(K,H) = (f^{-1} K' , f^{-1} H')$ for some $K' \rightsquigarrow H'$. Given $g \in G$ and $L \subset g H g^{-1}$, we must check that $( gKg^{-1} \cap L , L) \in (f_\subset)^{-1}(\rightsquigarrow)$. We have $f(gKg^{-1} \cap L) = f(g) K' f(g)^{-1} \cap fL$, where $f(g) \in G'$ and $fL \subset  f(g) H' f(g)^{-1}$. Since $K' \rightsquigarrow H'$ and $\rightsquigarrow$ is a transfer system, we also have $f(g) K' f(g)^{-1} \cap fL \rightsquigarrow fL$. Therefore $(f^{-1})_\subset(\rightsquigarrow) \subset \,\, \ra (f_\subset)^{-1}(\rightsquigarrow) \la$, and $f^{-1}_L(\rightsquigarrow) \subset f^{-1}_R(\rightsquigarrow)$ follows.
\end{proof}

Moreover, we can completely characterize when $f^{-1}_L = f^{-1}_R$. First, a lemma.

\begin{lem}\label{lem:injinvimts}Suppose that $m : G \to G'$ is an injective homomorphism between finite groups. Then for every $\rightsquigarrow \,\, \in \b{Tr}(G')$, the relation $(m_\subset)^{-1}(\rightsquigarrow)$ is a $G$-transfer system, and there is an equality $(m_\subset)^{-1}(\rightsquigarrow) = (m^{-1})_\subset(\rightsquigarrow)$.
\end{lem}

\begin{proof} As observed just prior to Definition \ref{defn:tsFiminv}, the relation $(m_\subset)^{-1}(\rightsquigarrow)$ is a partial order on $\b{Sub}(G)$ that refines inclusion. It is closed under conjugation because conjugation is preserved under images, and it is closed under restriction because if $(K,H) \in (m_\subset)^{-1}(\rightsquigarrow)$ and $L \subset H$, then $(m(K \cap L) , m(L) ) = (m(K) \cap m(L) , m(L) ) \in \,\, \rightsquigarrow$ because $m$ is injective. Thus $(m_\subset)^{-1}(\rightsquigarrow)$ is a $G$-transfer system. 

The inclusion $(m_\subset)^{-1}(\rightsquigarrow) \subset (m^{-1})_\subset(\rightsquigarrow)$ also follows from the injectivity of $m$, because if $(mK , mH) \in \,\, \rightsquigarrow$, then $(K,H) = (m^{-1}m K , m^{-1} m H) \in (m^{-1})_\subset(\rightsquigarrow)$. The other inclusion $(m_\subset)^{-1}(\rightsquigarrow) \supset (m^{-1})_\subset(\rightsquigarrow)$ holds because if $(K',H') \in \,\, \rightsquigarrow$, then $(m m^{-1} K' , m m^{-1} H') = (m(G) \cap K' , m(G) \cap H') \in \,\, \rightsquigarrow$ since $\rightsquigarrow$ is closed under restriction along $m(G) \cap H' \subset H'$.
\end{proof}

\begin{prop}\label{prop:injhomots} Suppose that $f : G \to G'$ is a homomorphism between finite groups. Then the following are equivalent:
	\begin{enumerate}
		\item{}$f$ is injective.
		\item{}$f^{-1}_L = f^{-1}_R$.
		\item{}$f^{-1}_L$ has a left adjoint.
	\end{enumerate}
Moreover, if $f$ is noninjective, then there is a strict inequality $f^{-1}_L(\rightsquigarrow) \subsetneq f^{-1}_R(\rightsquigarrow)$ for every $G'$-transfer system $\rightsquigarrow$.
\end{prop}

\begin{proof}$(1 \Rightarrow 2)$ follows from Lemma \ref{lem:injinvimts} and $(2 \Rightarrow 3)$ is immediate from the adjunction $f_L \dashv f^{-1}_R$. Now for $(3 \Rightarrow 1)$. Assume that $f$ is not injective. We shall show that $f^{-1}_L$ does not preserve all limits. For any $\rightsquigarrow \,\, \in \b{Tr}(G')$ and $(K,H) \in (f^{-1})_\subset(\rightsquigarrow)$, we have $\t{ker}(f) \subset K$. By part $(2)$ of \cite[Proposition B.4]{RubChar}, it follows that $H \not\subset \t{ker}(f)$ for every nontrivial relation $(K,H) \in f^{-1}_L(\rightsquigarrow)$. On the other hand, if $K \subset H \subset \t{ker}(f)$, then $(K,H) \in \,\, \ra (f_\subset)^{-1}(\rightsquigarrow) \la \,\, = f^{-1}_R(\rightsquigarrow)$ because for any $g \in G$ and $L \subset gHg^{-1}$, we have $f(gHg^{-1} \cap L) \rightsquigarrow fL$ since both sides are the trivial subgroup. It follows that $(1,\t{ker}(f)) \in f^{-1}_R(\rightsquigarrow) \setminus f^{-1}_L(\rightsquigarrow)$, which means the inclusion $f^{-1}_L(\rightsquigarrow) \subset f^{-1}_R(\rightsquigarrow)$ of Proposition \ref{prop:lrinvimineq} is strict. Therefore $f^{-1}_L$ does not preserve the terminal transfer system.
\end{proof}

\begin{cor}If $m : G \to G'$ is an injective homomorphism between finite groups, then there is a chain of order adjunctions $m_L \dashv m^{-1}_R = m^{-1}_L \dashv m_R$.
\end{cor}

If the homomorphism $f : G \to G'$ is non-injective, then the inequality $f^{-1}_L < f^{-1}_R$ reflects a pathology of the operadic $\t{ind} \dashv \t{res}$ adjunction (cf. part $(2)$ of Proposition \ref{prop:quilladjindrescoind}). Going forward, we shall only consider $f_L$ only when $f$ is injective.

\section{Operadic induction, restriction, and coinduction}\label{sec:indrescoind}

In this section, we relate images and inverse images of transfer systems to derived induction, restriction, and coinduction for marked operads. An important precedent to this work appears in \cite[\S6.2]{BH}, where Blumberg and Hill show how to calculate the admissible sets of a coinduced $N_\infty$ operad $\t{coind}_H^G \s{O}$ in terms of the admissible sets of $\s{O}$. We generalize to coinduction along a non-injective map, and we also analyze how restriction and induction behave for combinatorial operads. For any homomorphism $f : G \to G'$ between finite groups, we show that the adjunction $f^{-1}_L \dashv f_R$ of Definition \ref{defn:iminvimts} always lifts to derived restriction and coinduction for marked operads, and if $f$ is injective, we show that the adjunction $f_L \dashv f^{-1}_R = f^{-1}_L$ also lifts to derived induction and restriction (Theorem \ref{thm:iminvimcompatindrescoind}). On the other hand, if $f$ is noninjective, then we do not know how to make induction along $f$ homotopically meaningful because it is not a left Quillen functor (Proposition \ref{prop:quilladjindrescoind}). We briefly describe the situation for $N$ operads in \S\ref{subsec:Nopindrescoind}, and then we conclude by giving topological interpretations of our constructions in \S\ref{subsec:topinterpindrescoind}.

\subsection{Induction, restriction, and coinduction for marked operads}

A $G$-symmetric sequence $S \in \b{Sym}^G$ is the same thing as a nonequivariant symmetric sequence $S \in \b{Sym}$ equipped with a $G$-action through $\Sigma$-equivariant maps. Analogously, a marked $G$-operad $\s{O} \in \b{Op}^G_+$ is the same thing as a nonequivariant marked operad $\s{O} \in \b{Op}_+$ equipped with a $G$-action that preserves the operad structure and the markings. More formally, we have isomorphisms
	\[
	\b{Sym}^G \cong \b{Fun}(BG,\b{Sym})	\quad\t{and}\quad	\b{Op}^G_+ \cong \b{Fun}(BG,\b{Op}_+) ,
	\]
where $BG$ is the one-object category whose morphisms are the group $G$. This means we can define induction, restriction, and coinduction for marked operads and symmetric sequences using the usual Kan extension and pullback functors.

\begin{defn} Suppose that $f : G \to G'$ is a homomorphism between finite groups, and let $Bf : BG \to BG'$ for the corresponding functor on one-object categories. Define operadic induction, restriction, and coinduction functors by
	\begin{align*}
	\t{ind}_f	:=	\t{Lan}_{Bf} : \b{Op}^G_+ &\longrightarrow \b{Op}^{G'}_+	\\
		\b{Op}^{G}_+ &\longleftarrow \b{Op}^{G'}_+ : (Bf)^* =: \t{res}_f	\\
	\t{coind}_f	:=	\t{Ran}_{Bf} : \b{Op}^G_+	&\longrightarrow	\b{Op}^{G'}_+ ,
	\end{align*}
and similarly for symmetric sequences. The adjunctions $\t{ind}_f \dashv \t{res}_f \dashv \t{coind}_f$ follow formally from the universal properties of left and right Kan extension.
\end{defn}

The end and coend formulas imply that $\t{coind}_f$ and $\t{ind}_f$ are given by the familiar equalizers and coequalizers
	\[
	\t{coind}_f X \cong \t{eq} \Big( \prod_{G'} X \rightrightarrows \prod_{G'} \prod_G X \Big)
	\quad\t{and}\quad
	\t{ind}_f X \cong \t{coeq}\Big( \coprod_{G'} \coprod_{G} X \rightrightarrows \coprod_{G'} X \Big) ,
	\]
where $X$ is either an object of $\b{Sym}^G$ or $\b{Op}^G_+$, and all products and coproducts are taken in the corresponding category. In particular, the coproduct in $\b{Op}^G_+$ is an operadic wedge, which is analogous to an amalgamated free product of groups.

We shall derive the adjunctions $\t{ind}_f \dashv \t{res}_f \dashv \t{coind}_f$ using the model category structure for marked operads described in \S\ref{sec:background}. To that end, we must understand how $\t{ind}_f$ and $\t{res}_f$ interact with the generating cofibrations $F_+(\varnothing \to \frac{G \times \Sigma_n}{\Gamma})$ and $F_+(\frac{G \times \Sigma_n}{\Gamma} \to \frac{G \times \Sigma_n}{\Gamma} \sqcup \frac{G \times \Sigma_n}{\Gamma})$. Regarding $\b{Sym}^G$ and $\b{Op}^G_+$ as functor categories clarifies the matter. Let $F_+ : \b{Sym} \rightleftarrows \b{Op}_+ : U$ be the free-forgetful adjunction between nonequivariant symmetric sequences and marked operads.  Then for any finite group $G$, the induced adjunction
	\[
	F_+ \circ (-) : \b{Fun}(BG,\b{Sym}) \rightleftarrows \b{Fun}(BG,\b{Op}_+) : U \circ (-)
	\]
is isomorphic to the usual free-forgetful adjunction $F_+ : \b{Sym}^G \rightleftarrows \b{Op}^G_+ : U$, because the right adjoint forgets the operad structure in both cases. This implies the following commutation relations.

\begin{lem}\label{lem:indrescoindcomm} For any homomorphism $f : G \to G'$ between finite groups, there are natural isomorphisms
	\begin{align*}
	\t{ind}_f \circ F_+ \cong F_+ \circ \t{ind}_f	\quad&,\quad	\t{res}_f \circ F_+ \cong F_+ \circ \t{res}_f	\\
	\t{coind}_f \circ U \cong U \circ \t{coind}_f	\quad&,\quad	\t{res}_f \circ U \cong U \circ \t{res}_f ,
	\end{align*}
where $F_+ \dashv U$ denotes the free-forgetful adjunction between symmetric sequences and marked operads for either the group $G$ or the group $G'$.
\end{lem}

\begin{proof} The functor $\t{res}_f$ commutes with $F_+$ and $U$ because pre-composition commutes with post-composition, and the commutation relations for $\t{ind}_f$ and $\t{coind}_f$ follow from the uniqueness of adjoints.
\end{proof}

Thus, we are reduced to studying $\t{ind}_f$ and $\t{res}_f$ on symmetric sequences. Thinking of the components of a $G$-symmetric sequence as $(G \times \Sigma_n)$-sets, we have
	\[
	(\t{res}_f S')_n \cong \t{res}_{f \times \t{id}} S'_n	\quad\t{and}\quad	 (\t{ind}_f S)_n \cong \t{ind}_{f \times \t{id}} S_n
	\]
for every homomorphism $f : G \to G'$, $S \in \b{Sym}^G$, and $S' \in \b{Sym}^{G'}$. We arrive at the following result.

\begin{prop}\label{prop:quilladjindrescoind} Suppose that $f : G \to G'$ is an arbitrary homomorphism between finite groups. Then:
	\begin{enumerate}
		\item{} The adjunction $\t{res}_ f : \b{Op}^{G'}_+ \rightleftarrows \b{Op}^{G}_+: \t{coind}_f$ is a Quillen adjunction.
		\item{} The adjunction $\t{ind}_f :  \b{Op}^G_+ \rightleftarrows \b{Op}^{G'}_+ : \t{res}_f$ is a Quillen adjunction if and only if the homomorphism $f$ is injective.
	\end{enumerate}
\end{prop}

\begin{proof} We begin with $(1)$. Suppose that the morphism $i = F_+( \varnothing \to \frac{G' \times \Sigma_n}{\Gamma'})$ is a generating cofibration of $\b{Op}^{G'}_+$. By Lemma \ref{lem:indrescoindcomm}, there is an isomorphism $\t{res}_f( i) \cong F_+(\varnothing \to \t{res}_{f \times \t{id}} \frac{G' \times \Sigma_n}{\Gamma'})$, and the pulled back $G \times \Sigma_n$-set $\t{res}_{f \times \t{id}} \frac{G' \times \Sigma_n}{\Gamma'}$ is still $\Sigma_n$-free. Therefore there is a splitting
	\[
	\t{res}_{f \times \t{id}} \Bigg( \frac{G' \times \Sigma_n}{\Gamma'} \Bigg) \cong \coprod_{k = 1}^m \frac{G \times \Sigma_n}{\Gamma_k}
	\]
for some graph subgroups $\Gamma_1 , \dots , \Gamma_m \subset G \times \Sigma_n$. Since $F_+ : \b{Sym}^G \to \b{Op}^G_+$ preserves coproducts, we deduce that $\t{res}_f(i)$ is a coproduct of generating cofibrations, and hence a cofibration in $\b{Op}^G_+$. Inducting up relative cell complexes and passing to retracts proves that $\t{res}_f$ preserves all cofibrations. An analogous argument shows that $\t{res}_f$ also preserves acyclic cofibrations, and therefore $\t{res}_f \dashv \t{coind}_f$ is a Quillen adjunction.

Now for (2). Suppose first that $f$ is injective. Then we may assume $f : G \hookrightarrow G'$ is the inclusion of a subgroup, and that $\t{ind}_{f \times \t{id}} = \t{ind}_{G \times \Sigma_n}^{G' \times \Sigma_n}$ is induction in the usual sense. Arguing as above proves that $\t{ind}_f$ is left Quillen, because it preserves generating (acyclic) cofibrations.

Now suppose that $f$ is not injective. We shall show that $\t{ind}_f$ does not preserve all cofibrant operads. By \cite[Theorem 4.8]{RubComb}, it will be enough to find a cofibrant operad $\s{O} \in \b{Op}^G_+$ such that $\t{ind}_f \s{O}$ is not $\Sigma$-free. Suppose $\abs{G} = n$, and let $\Gamma$ be the graph of a permutation representation $\sigma : G \to \Sigma_{n}$ for $G/e$. Consider the operad $\s{O} = F_+(\frac{G \times \Sigma_{n}}{\Gamma})$, so that $\t{ind}_f \s{O} \cong F_+( \t{ind}_{f \times \t{id}} \frac{G \times \Sigma_{n}}{\Gamma})$. The symmetric sequence
	\[S = \t{ind}_{f \times \t{id}} \Bigg( \frac{G \times \Sigma_{n}}{\Gamma} \Bigg) \cong (G' \times \Sigma_n) \underset{G \times \Sigma_n}{\times} \Bigg( \frac{G \times \Sigma_n}{\Gamma} \Bigg)
	\]
is not $\Sigma$-free because the class $[(\t{id},\t{id}) , \Gamma]$ is fixed by $\sigma(\t{ker}f) \subset \Sigma_n$, and this subgroup is nontrivial because $\t{ker}f$ is nontrivial and $G$ acts faithfully on $G/e$. The operad $\t{ind}_f \s{O}$ also is not $\Sigma$-free, because there is a unit map $\eta : S \to \t{ind}_f \s{O}$.
\end{proof}

\begin{rem}\label{rem:derindrescoind} Every object of $\b{Op}^G_+$ is fibrant by \cite[Theorem 4.8]{RubComb}. Therefore $\t{coind}_f$ preserves all weak equivalences, which implies that $\bb{R}\t{coind}_f \cong \t{Ho}(\t{coind}_f)$ and $\bb{L}\t{res}_f \dashv \t{Ho}(\t{coind}_f)$. If $f$ is injective, then the functor $\t{res}_f$ also preserves all weak equivalences. In this case, we have isomorphisms $\bb{L}\t{res}_f \cong \t{Ho}(\t{res}_f) \cong \bb{R}\t{res}_f$ and a chain $\bb{L}\t{ind}_f \dashv \t{Ho}(\t{res}_f) \dashv \t{Ho}(\t{coind}_f)$ of derived adjunctions.
\end{rem}

\subsection{The connection to transfer systems} In this section, we relate derived operadic induction, restriction, and coinduction to image and inverse image constructions for transfer systems. Our strategy is to show that $\bb{L}\t{res}_f$ and $f^{-1}_L$ correspond under the equivalence $\t{Ho}(\b{Op}^G_+) \simeq \b{Tr}(G)$, and then to deduce the remaining correspondences from the uniqueness of adjoints.

Given that the left derived functor $\bb{L}\t{res}_f$ can be computed on free resolutions, and that $\t{res}_f$ commutes with $F_+ : \b{Sym}^G \to \b{Op}^G_+$, we are reduced to understanding the behavior of $\t{res}_f$ on symmetric sequences.

\begin{lem}\label{lem:dblcosetgraph}Suppose that $f : G \to G'$ is a homomorphism between finite groups, let $H \subset G'$ be a subgroup, and let $T$ be a $H$-set of finite cardinality $n$. Write $\Gamma(T)$ for the graph of a permutation representation of $T$. Then
	\[
	\t{res}_{f \times \t{id}} \Bigg( \frac{G' \times \Sigma_n}{\Gamma(T)} \Bigg)  \cong  \coprod_{r} \frac{G \times \Sigma_n} {\Gamma(f^* \t{res}^{r H r^{-1}}_{r H r^{-1} \cap \t{im}(f)} c_r T) } ,
	\]
where:
	\begin{enumerate}
		\item{}$r$ ranges over a set of representatives for $\t{im}(f) \times \Sigma_n \backslash G' \times \Sigma_n / \Gamma(T)$, taken in the subgroup $G' \times \{ \t{id} \}$, 
		\item{}$c_r T$ is the conjugate $r H r^{-1}$-action to $T$, and 
		\item{}$f^* \t{res}^{r H r^{-1}}_{r H r^{-1} \cap \t{im}(f)} c_r T$ is the $f^{-1}(r H r^{-1})$-action obtained by pulling back the $r H r^{-1} \cap \t{im}(f)$-action on $\t{res}^{r H r^{-1}}_{r H r^{-1} \cap \t{im}(f)} c_r T$ along $f$.
	\end{enumerate}
\end{lem}

\begin{proof} Compute $\t{res}_{f \times \t{id}}$ by first restricting to the subgroup $\t{im}(f) \times \Sigma_n \subset G' \times \Sigma_n$ and applying the double-coset formula, and then pulling back along the surjective homomorphism $f \times \t{id} : G \times \Sigma_n \to \t{im}(f) \times \Sigma_n$.

The first step yields
	\[
	\coprod_{r}
	\frac{\t{im}(f) \times \Sigma_n} { r \Gamma(T) r^{-1} \cap ( \t{im}(f) \times \Sigma_n ) } ,
	\]
where $r$ ranges over a set of representatives for $\t{im}(f) \times \Sigma_n \backslash G' \times \Sigma_n / \Gamma(T)$. We may assume $r \in G' \times \{\t{id}\}$ because we are taking $\t{im}(f) \times \Sigma_n$-orbits. Moroever, we have $r \Gamma(T) r^{-1} = \Gamma(c_r T)$, and $\Gamma(c_r T) \cap ( \t{im}(f) \times \Sigma_n ) = \Gamma(\t{res}^{rHr^{-1}}_{rHr^{-1} \cap \t{im}(f)} c_r T)$.

The second step yields
	\[
	\coprod_{r}
	(f \times \t{id})^* \Bigg( \frac{\t{im}(f) \times \Sigma_n}  {\Gamma(\t{res}^{rHr^{-1}}_{rHr^{-1} \cap \t{im}(f)} c_r T) } \Bigg) ,
	\]
and each summand is a transitive $(G \times \Sigma_n)$-set because $f \times \t{id} : G \times \Sigma_n \to \t{im}(f) \times \Sigma_n$ is surjective. Since stabilizers pull back, it follows
	\[
	(f \times \t{id})^* \Bigg( \frac{\t{im}(f) \times \Sigma_n}  {\Gamma(\t{res}^{rHr^{-1}}_{rHr^{-1} \cap \t{im}(f)} c_r T) } \Bigg) 
	\cong  \frac{G \times \Sigma_n}{(f \times \t{id})^{-1} \Gamma(\t{res}^{rHr^{-1}}_{rHr^{-1} \cap \t{im}(f)} c_r T) } ,
	\]
and $(f \times \t{id})^{-1} \Gamma(\t{res}^{rHr^{-1}}_{rHr^{-1} \cap \t{im}(f)} c_r T) = \Gamma(f^*\t{res}^{rHr^{-1}}_{rHr^{-1} \cap \t{im}(f)} c_r T)$.
\end{proof}

From here, we can calculate the transfer system associated to ${\bb{L}\t{res}_f \s{O}}$.

\begin{thm}\label{thm:iminvimcompatindrescoind}Suppose that $f : G \to G'$ is an arbitrary homomorphism between finite groups. Then the squares
	\[
	\begin{tikzpicture}[scale=2]
		\node(00) at (0,0) {$\b{Tr}(G)$};
		\node(01) at (0,0.7) {$\t{Ho}(\b{Op}^{G}_+)$};
		\node(10) at (1.75,0) {$\b{Tr}(G')$};
		\node(11) at (1.75,0.7) {$\t{Ho}(\b{Op}^{G'}_+)$};
		
		\path[->]
		(10) edge [below] node {$f^{-1}_L$} (00)
		(11) edge [above] node {$\bb{L}\t{res}_f$} (01)
		(01) edge [left] node {$\to_\bullet$} (00)
		(11) edge [right] node {$\to_\bullet$} (10)
		;
		
		\node(00') at (3.25,0) {$\b{Tr}(G)$};
		\node(01') at (3.25,0.7) {$\t{Ho}(\b{Op}^{G}_+)$};
		\node(10') at (5,0) {$\b{Tr}(G')$};
		\node(11') at (5,0.7) {$\t{Ho}(\b{Op}^{G'}_+)$};
		
		\path[->]
		(00') edge [below] node {$f_R$} (10')
		(01') edge [above] node {$\t{Ho}(\t{coind}_f)$} (11')
		(01') edge [left] node {$\to_\bullet$} (00')
		(11') edge [right] node {$\to_\bullet$} (10')
		;
	\end{tikzpicture}
	\]
commute. Suppose additionally that the map $f$ is injective. Then the squares
	\[
	\begin{tikzpicture}[scale=2]
		\node(00) at (0,0) {$\b{Tr}(G)$};
		\node(01) at (0,0.7) {$\t{Ho}(\b{Op}^{G}_+)$};
		\node(10) at (1.75,0) {$\b{Tr}(G')$};
		\node(11) at (1.75,0.7) {$\t{Ho}(\b{Op}^{G'}_+)$};
		
		\path[->]
		(00) edge [below] node {$f_L$} (10)
		(01) edge [above] node {$\bb{L}\t{ind}_f$} (11)
		(01) edge [left] node {$\to_\bullet$} (00)
		(11) edge [right] node {$\to_\bullet$} (10)
		;
		
		\node(00') at (3.25,0) {$\b{Tr}(G)$};
		\node(01') at (3.25,0.7) {$\t{Ho}(\b{Op}^{G}_+)$};
		\node(10') at (5,0) {$\b{Tr}(G')$};
		\node(11') at (5,0.7) {$\t{Ho}(\b{Op}^{G'}_+)$};
		
		\path[->]
		(10') edge [below] node {$f^{-1}_R = f^{-1}_L$} (00')
		(11') edge [above] node {$\t{Ho}(\t{res}_f) \cong \bb{L}\t{res}_f$} (01')
		(01') edge [left] node {$\to_\bullet$} (00')
		(11') edge [right] node {$\to_\bullet$} (10')
		;
	\end{tikzpicture}
	\]
also commute.	
\end{thm}

\begin{proof} We begin by checking the equation $\to_\bullet \circ \,\, \bb{L}\t{res}_f = f^{-1}_L \,\, \circ \to_\bullet$ on operads. Suppose that $\s{O} \in \b{Op}^{G'}_+$, and let
	\[
	Q\s{O} = F_+ \Bigg( \coprod_{ H/K \in A(\s{O}) }
	 \frac{G' \times \Sigma_{|H:K|}}{ \Gamma(H/K)} \Bigg) ,
	\]
where $H/K$ ranges over all orbits in $A(\s{O})$. Then $Q\s{O}$ is a cofibrant replacement for $\s{O}$, because choosing $\Gamma(H/K)$-fixed operations in $\s{O}$ determines a map $Q\s{O} \to \s{O}$, and this map is a weak equivalence because $A(Q\s{O}) = \la H/K \,| \, H/K \in A(\s{O}) \ra = A(\s{O})$ by \cite[Theorem 5.6]{RubComb}. By Lemmas \ref{lem:indrescoindcomm} and \ref{lem:dblcosetgraph}, we conclude that 
	\[
	\t{res}_f(Q\s{O}) \cong F_+ \Bigg( \coprod_{H/K \in A(\s{O})} \coprod_{r} \frac{G \times \Sigma_{\abs{H:K}}}{\Gamma(f^*\t{res}^{rHr^{-1}}_{rHr^{-1} \cap \t{im}(f)} c_r H/K)} \Bigg) ,
	\]
where $r$ ranges over the representatives specified in Lemma \ref{lem:dblcosetgraph} for each orbit $H/K$.

By \cite{RubComb} once more, the class of admissible sets of $\t{res}_f(Q\s{O})$ is the indexing system
	\[
	\Bigg\langle f^* \t{res}^{rHr^{-1}}_{rHr^{-1} \cap \t{im}(f)} c_r H/K  \,\, \Bigg|  \,\, 
	\begin{array}{c}
		H/K \in A(\s{O}) \t{ and }	\\
		r \in \t{im}(f) \times \Sigma_{|H:K|} \backslash G' \times \Sigma_{|H:K|} / \Gamma(H/K)
	\end{array} \Bigg\rangle .
	\]
Since indexing systems are closed under conjugation, this simplifies to
	\[
	\Big\langle f^*\t{res}^H_{H \cap \t{im}(f)} H/K \, \Big| \, H/K \in A(\s{O}) \Big\rangle ,
	\]
and since indexing systems are closed restriction and subobjects, and $f^*$ commutes with coproducts, this simplifies further to
	\[
	\Big\langle f^*H/K \, \Big| \, H/K \in A(\s{O}) \t{ and } H \subset \t{im}(f) \Big\rangle = \Big\langle f^{-1}H/f^{-1}K \, \Big| \, H/K \in A(\s{O}) \Big\rangle .
	\]
This computes $A(\t{res}_f(Q\s{O}))$. Applying the isomorphism $\to_\bullet \, : \b{Ind}(G) \to \b{Tr}(G)$ and \cite[Proposition B.9]{RubChar} shows that the transfer system associated to $\t{res}_f(Q\s{O})$ is $\la (f^{-1}K , f^{-1}H) \, | \, K \to_\s{O} H \rangle$, which equals $f^{-1}_L(\to_{\s{O}})$ by definition. This proves that $\to_{\bb{L}\t{res}_f \s{O}} \,\, = f^{-1}_L(\to_{\s{O}})$ for every operad $\s{O} \in \b{Op}^{G'}_+$, and the equality $\to_\bullet \circ \,\,  \bb{L}\t{res}_f = f^{-1}_L \,\, \circ \to_\bullet$ of functors follows because parallel morphisms in $\b{Tr}(G)$ are equal.

Now let $\to_\bullet^{-1}$ be a pseudoinverse to $\to_\bullet$. The equation $\to_\bullet \circ \,\,  \bb{L}\t{res}_f = f^{-1}_L \,\, \circ \to_\bullet$ implies an isomorphism $\t{Ho}(\t{coind}_f) \,\, \circ \to_\bullet^{-1} \,\, \cong \,\, \to_\bullet^{-1} \circ \,\, f_R$ of right adjoints, and hence $\to_\bullet \circ \,\, \t{Ho}(\t{coind}_f) \cong f_R \,\, \circ \to_\bullet$ as well. Since $\b{Tr}(G')$ is a poset, this is an equality. 

Suppose further that the morphism $f : G \to G'$ is injective. Then $f^{-1}_L = f^{-1}_R$ by Proposition \ref{prop:injhomots}, and $\bb{L}\t{res}_f \cong \t{Ho}(\t{res}_f) \cong \bb{R}\t{res}_f$ by Remark \ref{rem:derindrescoind}. Our calculation of $\to_{\bb{L}\t{res}_f\s{O}}$ now reads $\to_\bullet \circ \,\, \t{Ho}(\t{res}_f) = f^{-1}_R \,\, \circ \to_\bullet$, and the equality $\to_\bullet \circ \,\, \bb{L}\t{ind}_f = f_L \,\, \circ \to_\bullet$ for left adjoints follows as above.
\end{proof}

The functor $\t{res}_f$ is already homotopical when $f$ is injective (cf. Remark \ref{rem:derindrescoind}), but even when $f$ is not, the next result shows that $\t{res}_f$ still preserves weak equivalences in the most interesting cases. Thus $\t{res}_f$ barely needs to be derived.

\begin{cor}\label{cor:invimresNinfty}Suppose $f : G \to G'$ is an arbitrary homomorphism between finite groups. Then the functor $\t{res}_f : \b{Op}^{G'}_+ \to \b{Op}^G_+$ preserves weak equivalences between $N$ operads. Moreover, if $\s{O} \in \b{Op}^{G'}_+$ is a $N$ operad, then so is $\t{res}_f \s{O}$, and the equality $\to_{\t{res}_f \s{O}}\,\, = f^{-1}_L(\to_{\s{O}})$ holds.
\end{cor}

\begin{proof}Suppose $\s{O},\s{P} \in \b{Op}^{G'}_+$ are $N$ operads and that $\vp : \s{O} \to \s{P}$ is a weak equivalence. Then for any $n \geq 0$ and any subgroup $\Xi \subset G' \times \Sigma_n$ whatsoever, the set $\s{O}(n)^{\Xi}$ is nonempty if and only if the set $\s{P}(n)^{\Xi}$ is nonempty. When $\Xi$ is a graph subgroup, this follows from the definition of a weak equivalence. When $\Xi$ is not, both sides are empty. The restricted operads $\t{res}_f \s{O}$ and $\t{res}_f \s{P}$ have the same property because $(\t{res}_f\s{O})(n)^{\Xi} = \s{O}(n)^{(f \times \t{id})\Xi}$, and therefore the map $\t{res}_f \vp : \t{res}_f \s{O} \to \t{res}_f \s{P}$ is also a weak equivalence.

Now suppose $\s{O} \in \b{Op}^{G'}_+$ is a $N$ operad. Then $\t{res}_f \s{O}$ is $\Sigma$-free because it has the same $\Sigma$-action, and for any $n \geq 0$, we have $(\t{res}_f \s{O})(n)^G = \s{O}(n)^{f(G)} \supset \s{O}(n)^{G'} \neq \varnothing$. Therefore $\t{res}_f \s{O}$ is also a $N$ operad. Choose a cofibrant replacement $q : Q\s{O} \to \s{O}$. Then $\t{res}_f(q) : \bb{L}\t{res}_f \s{O} \simeq \t{res}_f(Q\s{O}) \to \t{res}_f \s{O}$ is a weak equivalence by the preceding paragraph, and hence $\to_{\t{res}_f \s{O}} \,\, = \,\, \to_{\bb{L}\t{res}_f \s{O}} \,\, = f^{-1}_L(\to_{\s{O}})$.
\end{proof}

We also have the following consistency check for Theorem \ref{thm:iminvimcompatindrescoind}.

\begin{ex} Suppose that $f : G \to G'$ is a homomorphism between finite groups, that $\s{O} \in \b{Op}^G_+$ is a marked $G$-operad, and that $\to_{\s{O}}$ is the terminal $G$-transfer system. Then $\to_{\t{coind}_f \s{O}} \,\, = f_R(\to_{\s{O}})$ is also terminal, because $f_R$ is a right adjoint. When $f$ is the unique map $! : 1 \to G$, and $\s{O} = \b{As}$, we conclude that the transfer system for $\t{coind}_{1}^G(\b{As}) \cong \b{Set}(G,\b{As})$ is terminal, just as in Example \ref{ex:coindlat}.

More generally, Theorem \ref{thm:iminvimcompatindrescoind} says that the transfer system $\to$ associated to $\b{Set}(H\backslash G,\b{As}) \cong \t{coind}_H^G(\b{As})$ equals $i_R(=)$, where $i : H \hookrightarrow G$ is the inclusion and $=$ is the trivial $H$-transfer system. By definition, $J \to K$ if and only if $gJg^{-1} \cap L \cap H = L \cap H$ for every $g \in G$ and $L \subset gKg^{-1}$, which is equivalent to requiring $K \cap \bigcup_{g \in G} g^{-1} H g \subset J$. Since $K$ and $\bigcup_{g \in G} g^{-1} H g$ are stable under conjugation by elements of $K$, this is equivalent to the inclusion
	\[
	\bigcup_{g \in G} \t{Stab}_K(Hg \in H \backslash G) = \bigcup_{g \in G} (K \cap g^{-1} H g) \subset \bigcap_{k \in K} kJk^{-1} = \bigcap_{k \in K} \t{Stab}_K(kJ \in K/J),
	\]
which says that every element of $K$ that fixes an element of $H \backslash G$ acts as the identity on $K/J$. This recovers the description of the $A(\b{Set}(H \backslash G,\b{As}))$ in Example \ref{ex:coindlat}.
\end{ex}

\subsection{The unmarked case}\label{subsec:Nopindrescoind} If $G$ is a finite group and $H \subset G$ is a subgroup, then we also have adjoint functors $\t{ind}_H^G \dashv \t{res}^G_H \dashv \t{coind}_H^G$ between the categories $\b{Op}^H$ and $\b{Op}^G$ of unmarked operads. Unfortunately, the previous discussion does not entirely carry over, because the functor $\t{ind}_H^G : \b{Op}^H \to \b{Op}^G$ does not preserve $N$ operads. For example, if $\s{O} = F( \frac{H \times \Sigma_0}{H} \sqcup \frac{H \times \Sigma_2}{H})$, then $\t{ind}_H^G \s{O} \cong F ( \frac{G \times \Sigma_0}{H} \sqcup \frac{G \times \Sigma_2}{H})$, and this operad does not have any $G$-fixed operations. Replacing the category $\b{Op}^G$ with $\b{Op}^G_+$ fixes this problem because we change the coproduct.

That being said, there are no issues with using unmarked operads if one is only concerned with restriction and coinduction.

\begin{lem}Suppose $f : G \to G'$ is a homomorphism between finite groups, and consider the adjunction $\t{res}_f : \b{Op}^{G'} \rightleftarrows \b{Op}^G : \t{coind}_f$. Both adjoints preserve $N$ operads and weak equivalences between $N$ operads, and therefore there is an induced adjunction $\t{Ho}(\t{res}_f) : \t{Ho}(N\t{-}\b{Op}^{G'}) \rightleftarrows \t{Ho}(N\t{-}\b{Op}^G) : \t{Ho}(\t{coind}_f)$.
\end{lem}

\begin{proof}The proof of Corollary \ref{cor:invimresNinfty} shows that $\t{res}_f$ preserves $N$ operads and weak equivalences between them. We must show that $\t{coind}_f$ has the same two properties.

Suppose $\s{O}$ is a $N$ $G$-operad. Then $\t{coind}_f \s{O}(n)$ is the set of $G$-equivariant maps $\alpha : G' \to \s{O}(n)$, where $G$ acts on $G'$ on the left through $f : G \to G'$. The $G' \times \Sigma_n$-action is given by $(\alpha \cdot \sigma)(x) = \alpha(x) \cdot \sigma$ and $(g' \cdot \alpha)(x) = \alpha(x \cdot g')$, for any $\sigma \in \Sigma_n$ and $g' \in G'$. Thus, the $\Sigma$-freeness of $\s{O}$ implies the $\Sigma$-freeness of $\t{coind}_f \s{O}$ by evaluating at some $x \in G'$, and the constant function $c_y : G' \to \s{O}(n)$ valued at any $y \in \s{O}(n)^G$ is a $G'$-fixed element of $\t{coind}_f \s{O}(n)$. Therefore $\t{coind}_f \s{O}$ is a $N$ $G'$-operad.

Now suppose $\vp : \s{O} \to \s{P}$ is a weak equivalence between $N$ $G$-operads and consider $\t{coind}_f \vp : \t{coind}_f \s{O} \to \t{coind}_f \s{P}$. We must show that for any $n \geq 0$ and graph subgroup $\Gamma \subset G' \times \Sigma_n$, if $\t{coind}_f \s{P}(n)^\Gamma \neq \varnothing$, then $\t{coind}_f \s{O}(n)^\Gamma \neq \varnothing$. An element of $\t{coind}_f \s{P}(n)^\Gamma$ is represented by a $G' \times \Sigma_n$-map $\frac{G' \times \Sigma_n}{\Gamma} \to \t{coind}_f \s{P}(n)$, which is adjoint to a $G \times \Sigma_n$-map $\t{res}_f(\frac{G' \times \Sigma_n}{\Gamma}) \to \s{P}(n)$. As in the proof of Proposition \ref{prop:quilladjindrescoind}, the $G \times \Sigma_n$-set $\t{res}_f(\frac{G' \times \Sigma_n}{\Gamma})$ splits as a coproduct $\coprod_{k=1}^m \frac{G \times \Sigma_n}{\Gamma_k}$ for some graph subgroups $\Gamma_1,\dots,\Gamma_k \subset G \times \Sigma_n$. Since $\vp : \s{O} \to \s{P}$ is a weak equivalence, each component map $\frac{G \times \Sigma_n}{\Gamma_k} \to \s{P}(n)$ lifts to $\s{O}(n)$. Summing up these lifts and applying the adjunction gives a $\Gamma$-fixed point of $\t{coind}_f \s{O}(n)$.

It follows that there is an adjunction $\t{res}_f : N\t{-}\b{Op}^{G'} \rightleftarrows N\t{-}\b{Op}^G : \t{coind}_f$, and since both adjoints are homotopical, the adjunction descends to homotopy categories (e.g. through a trivial application of \cite[\S44.2]{DHKS}).
\end{proof}

We obtain an unmarked analogue to Theorem \ref{thm:iminvimcompatindrescoind}.

\begin{prop}\label{prop:unmarkrescoind} Suppose that $f : G \to G'$ is an arbitrary homomorphism between finite groups. Then the squares below commute.
	\[
	\begin{tikzpicture}[scale=2]
		\node(00) at (0,0) {$\b{Tr}(G)$};
		\node(01) at (0,0.7) {$\t{Ho}(N\t{-}\b{Op}^G)$};
		\node(10) at (1.75,0) {$\b{Tr}(G')$};
		\node(11) at (1.75,0.7) {$\t{Ho}(N\t{-}\b{Op}^{G'})$};
		
		\path[->]
		(10) edge [below] node {$f^{-1}_L$} (00)
		(11) edge [above] node {$\t{Ho}(\t{res}_f)$} (01)
		(01) edge [left] node {$\to_\bullet$} (00)
		(11) edge [right] node {$\to_\bullet$} (10)
		;
		
		\node(00') at (3.25,0) {$\b{Tr}(G)$};
		\node(01') at (3.25,0.7) {$\t{Ho}(N\t{-}\b{Op}^G)$};
		\node(10') at (5,0) {$\b{Tr}(G')$};
		\node(11') at (5,0.7) {$\t{Ho}(N\t{-}\b{Op}^{G'})$};
		
		\path[->]
		(00') edge [below] node {$f_R$} (10')
		(01') edge [above] node {$\t{Ho}(\t{coind}_f)$} (11')
		(01') edge [left] node {$\to_\bullet$} (00')
		(11') edge [right] node {$\to_\bullet$} (10')
		;
	\end{tikzpicture}
	\]
\end{prop}

\begin{proof} As in the proof of Theorem \ref{thm:iminvimcompatindrescoind}, it is enough to show $\to_{\t{res}_f \s{O}} \,\, =  f^{-1}_L(\to_{\s{O}})$ for any $\s{O} \in N\t{-}\b{Op}^{G'}$. Choose markings $u \in \s{O}(0)^{G'}$ and $\s{O}(2)^{G'}$ and regard $\s{O}$ as a $N$ operad in $\b{Op}^{G'}_+$. Then Corollary \ref{cor:invimresNinfty} gives the desired result.
\end{proof}

\subsection{Topological interpretations}\label{subsec:topinterpindrescoind} As in \S\ref{subsec:topprodcopten}, we think of induction, restriction, and coinduction for $N$ operads and marked $G$-operads as the homotopically correct constructions, and then we use the functor $\bb{E} : \b{Set} \to \b{Top}$ from \S\ref{sec:background} to push things into topology. This section describes the results.

We begin with restriction and coinduction. In this case, it is simplest to model $N_\infty$ operads as $N$ operads via the functors $\bb{E} : N\t{-}\b{Op}^G \rightleftarrows N_\infty\t{-}\b{Op}^G : (-)^u$ from Proposition \ref{prop:Eu}. We temporarily introduce the following definitions.

\begin{defn}Let $f: G \to G'$ be a homomorphism between finite groups, and suppose that $\s{O}$ is a $N_\infty$ $G$-operad and that $\s{O}'$ is a $N_\infty$ $G'$-operad. Define $N_\infty$ restriction and coinduction by
	\[
	\t{res}_f^{N_\infty} \s{O}' = \bb{E}(\t{res}_f(\s{O}'^u)) \quad\t{and}\quad \t{coind}_f^{N_\infty} \s{O} = \bb{E}(\t{coind}_f(\s{O}^u)) ,
	\]
where $\t{res}_f$ and $\t{coind}_f$ denote ordinary operadic restriction and coinduction.
\end{defn}

These derived constructions agree with the ordinary ones.

\begin{prop} Suppose $f : G \to G'$ is a homomorphism between finite groups.
	\begin{enumerate}
		\item{}For any $N_\infty$ $G'$-operad $\s{O}'$, there is an equivalence $\t{res}_f \s{O}' \simeq \t{res}_f^{N_\infty}\s{O}'$, and therefore $\to_{\t{res}_f \s{O}'} \,\, = f^{-1}_L(\to_{\s{O}'})$.
		\item{}For any $N_\infty$ $G$-operad $\s{O}$, there is an equivalence $\t{coind}_f \s{O} \simeq \t{coind}_f^{N_\infty} \s{O}$, and therefore $\to_{\t{coind}_f \s{O}} \,\, = f_R(\to_{\s{O}})$.
	\end{enumerate}
\end{prop}

\begin{proof} We begin with restriction. The functor $(-)^u$ commutes with restriction, and $\t{res}_f$ preserves $N_\infty$ operads. Therefore $\t{res}_f^{N_\infty} \s{O}' \cong \bb{E}((\t{res}_f \s{O}')^u) \simeq \t{res}_f \s{O}'$ and
	\[
	\to_{\t{res}_f \s{O}'} \,\, = \,\, \to_{\bb{E}(\t{res}_f(\s{O}'^u))} \,\, = \,\, \to_{\t{res}_f(\s{O}'^u)} \,\, = f^{-1}_L( \to_{\s{O}'^u} ) = f^{-1}_L( \to_{\s{O}'} )
	\]
by Proposition \ref{prop:unmarkrescoind}. For part (2), note that the functor $(-)^u$ commutes with coinduction because the forgetful functor $U : \b{Top} \to \b{Set}$ from the category of compactly generated weak Hausdorff spaces preserves limits. Coinduction also preserves $N_\infty$ operads, and now we may argue as before.
\end{proof}

The situation for $N_\infty$ induction is more complicated. To make sense of the construction, we model $N_\infty$ $G$-operads as marked $G$-operads in $\b{Op}^G_+$. There is a DK equivalence $\bb{LE} = \bb{E} \circ Q : \b{Op}^G_+ \to N_\infty\t{-}\b{Op}^G$, where $Q$ denotes cofibrant replacement in $\b{Op}^G_+$, and the functor $F_+ \circ (-)^u : N_\infty\t{-}\b{Op}^G \to \b{Op}^G_+$ is inverse to $\bb{LE}$ up to zig-zags of natural weak equivalences. This justifies the next definition.

\begin{defn} Suppose $f : G \hookrightarrow G'$ is an injective homomorphism between finite groups. For any $N_\infty$ $G$-operad $\s{O}$, let
	\[
	\t{ind}_f^{N_\infty} \s{O} = \bb{E}(\t{ind}_f F_+(\s{O}^u)) ,
	\]
where $\t{ind}_f$ denotes ordinary operadic induction.
\end{defn}

The operad $\t{ind}_f^{N_\infty} \s{O}$ has the desired homotopy type, because the cofibrancy of $F_+(\s{O}^u)$ implies there is an equivalence
	\[
	\bb{E} (\t{ind}_f F_+(\s{O}^u)) \simeq  \bb{LE} (\bb{L} \t{ind}_f F_+(\s{O}^u)),
	\]
and therefore $\to_{\t{ind}_f^{N_\infty}\s{O}} \,\, = f_L (\to_{\s{O}})$ by Theorem \ref{thm:iminvimcompatindrescoind}. Unfortunately, the operad $\t{ind}_f^{N_\infty} \s{O}$ is quite far from the ordinary induced operad $\t{ind}_f \s{O}$, and it seems difficult to induce an action by a $N_\infty$ $H$-operad up to an action by a $N_\infty$ $G$-operad in general. The basic issue is illustrated below.

\begin{ex} Induction is an indexed coproduct, so we shall elaborate on Example \ref{ex:Ninftycop}. Suppose $i : H \hookrightarrow G$ is the inclusion of a subgroup, $X$ is a $G$-space, and $\s{O}$ is a $N_\infty$ $H$-operad. An action of $\t{ind}_i \s{O}$ on $X$ is equivalent to an action of $\s{O}$ on $\t{res}_i X$, and if $f \in \s{O}(n)$ represents an operation $F$ on $X$, then $\s{O}$ parametrizes coherence homotopies between all $H \times \Sigma_n$-conjugates of $F$. On the other hand, an action of $\t{ind}_i^{N_\infty}\s{O}$ must parametrize coherence homotopies between all $G \times \Sigma_n$-conjugates of $F$. The $\s{O}$-action gives homotopies between sets of $g(H \times \Sigma_n)$-conjugates for each $g \in G$, but nothing between $(g,\sigma) \cdot F$ and $(g',\sigma') \cdot F$ if $(g,\sigma)$ and $(g',\sigma')$ are in different $H \times \Sigma_n$-cosets of $G \times \Sigma_n$.

As before, there is more hope if we work in a marked setting. If $\s{O}$ has a distinguished unit $u \in \s{O}(0)^H$ and product $p \in \s{O}(2)^H$ that represent $G$-fixed operations over $X$, then $\s{O}$ specifies a homotopy between $F$ and $P(\dots P(P(x_1,x_2),x_3), \dots , x_n)$, which conjugates to a homotopy between $g \cdot F$ to $P(\dots P(P(x_1,x_2),x_3), \dots , x_n)$ for any $g \in G$. We can concatenate these two homotopies just as we did in Example \ref{ex:Ninftycop}, but this composite might not have the right equivariance.
\end{ex}

\appendix
\section{Quotient operads}\label{app:quotop}

If $\s{O}$ is an operad and $\sim$ is a congruence relation on $\s{O}$, then identifying $\ol{\s{O}} = \s{O}/\!\sim$ typically amounts to solving a word problem. In general, these problems can be quite complicated, but we can gain traction in a few cases by introducing a ``direction'' to the relation $\sim$. In this appendix, we explain how to use this technique to identify quotient operads (Propositions \ref{prop:idquot} and \ref{prop:idquotcrit}), and then we analyze the coproduct $\s{O} * \s{P}$ of $N$ operads (Example-Lemma \ref{exlem:Nopcop}) and the tensor product $\s{O} \otimes \s{P}$ of free $G$-operads (Lemma \ref{lem:freeopten}).

\subsection{Solving operadic word problems}\label{subsec:solveword} Throughout this section, we assume that $\s{O}$ is an operad in $\b{Set}^G$, $c : \coprod_{n \geq 0} \s{O}(n) \to \bb{N}$ is a function, and $R = (R_n)_{n \in \bb{N}}$ is a graded binary relation on $\s{O}$ such that for any integer $n \geq 0$ and operations $f,f' \in \s{O}(n)$, if $f R f'$, then $c(f) > c(f')$. We think of $c$ as a complexity function and $R$ as a complexity-reducing relation. In practice, $\s{O}$ will be a free operad whose operations $f \in \s{O}$ are formal composites, $c(f)$ will be a weighted count of the terms in a composite $f \in \s{O}$, and $R$ will indicate a reduction of one composite $f$ into another composite $f'$. Accordingly, we introduce some terminology.

\begin{defn} An operation $f \in \s{O}$ is \emph{reduced} if there is no $f' \in \s{O}$ such that $f R f'$. An operation $h$ is a \emph{reduced form} of the operation $f$ if:
	\begin{enumerate}[label=(\alph*)]
		\item{}the operation $h$ is reduced, and
		\item{} there is a chain $f = f_0 R f_1 R \cdots R f_n = h$ of $R$-relations connecting $f$ to $h$.
	\end{enumerate}
The case $n=0$ is allowed, in which case the chain reads $f = f_0 = h$.
\end{defn}

Given $\s{O}$, $c$, and $R$ as above, we would like to say that:
	\begin{enumerate}
		\item{}every operation $f \in \s{O}$ has a unique reduced form $\ol{f}$, and
		\item{}the reduced operations in $\s{O}$ are a set of representatives for the congruence relation $\la R \ra$ that $R$ generates.
	\end{enumerate}
For our purposes, it will also be convenient if
	\begin{enumerate}[resume]
		\item{}the set $r\s{O}$ of reduced operations forms a sub-$G$-symmetric sequence of $\s{O}$. 
	\end{enumerate}
If these three conditions hold, then we can easily identify the quotient $\ol{\s{O}} = \s{O}/\la R \ra$.

\begin{prop}\label{prop:idquot} Assume that $\s{O}$ is a $N$ operad, and that conditions (1)--(3) hold. Then the underlying symmetric sequence of $\ol{\s{O}} = \s{O}/\la R \ra$ is isomorphic to $r\s{O}$, and equipping $r\s{O}$ with the operad structure
	\[
	\gamma_{r\s{O}}(f;h_1,\dots,h_n) = \ol{\gamma_{\s{O}}(f;h_1,\dots,h_n)} \quad\t{and}\quad \t{id}_{r\s{O}} = \ol{\t{id}_{\s{O}}}.
	\]
makes $r\s{O}$ and $\ol{\s{O}}$ isomorphic as operads. It follows that $\ol{\s{O}}$ is a $N$ operad with the same admissible sets as $\s{O}$.
\end{prop}

\begin{proof} Consider the composite $\vp : r\s{O} \hookrightarrow U\s{O} \twoheadrightarrow U\ol{\s{O}}$ of the inclusion and the quotient. By (3), this is a map of symmetric sequences. By (1) and (2), the unique reduced representative of a class $[f] \in \ol{\s{O}}$ is $\ol{f}$, and therefore $\vp$ has an inverse given by the formula $\vp^{-1}[f] = \ol{f}$. Therefore $r\s{O} \cong U\ol{\s{O}}$, and we translate the operad structure from $\ol{\s{O}}$ to $r\s{O}$ using $\vp$ and $\vp^{-1}$.

We have a sequence $U\s{O} \twoheadrightarrow U\ol{\s{O}} \cong r\s{O} \hookrightarrow U\s{O}$ of maps of symmetric sequences. The quotient map $U\s{O} \twoheadrightarrow U\ol{\s{O}}$ ensures that $\ol{\s{O}}(n)^G$ is nonempty for all $n \geq 0$, and the inclusion map $U\ol{\s{O}} \hookrightarrow U\s{O}$ ensures that $\ol{\s{O}}$ is $\Sigma$-free. Therefore $\ol{\s{O}}$ is a $N$ operad. As for its admissible sets, the quotient map implies $A(\s{O}) \subset A(\ol{\s{O}})$, and the inclusion map implies $A(\ol{\s{O}}) \subset A(\s{O})$.
\end{proof}

Conditions (1)--(3) are not automatic, but we can enforce them by placing a few assumptions on the relation $R$.

\begin{prop}\label{prop:idquotcrit} Assume that $R$ has the following four properties:
	\begin{enumerate}[label=(\roman*)]
		\item{}For any integer $n \geq 0$ and operations $f,h,h' \in \s{O}(n)$, if $fRh$ and $fRh'$, then there is an operation $k \in \s{O}(n)$ and a pair of coterminal chains $h = h_0 R h_1 R \cdots R h_m = k$ and $h' = h'_0 R h'_1 R \cdots R h'_{m'} = k$ with $m,m' \geq 0$.
		\item{}For any integer $n \geq 0$, operations $f,f' \in \s{O}(n)$, and group elements $g \in G$ and $\sigma \in \Sigma_n$, if $f R f'$, then $(gf\sigma) R (gf'\sigma)$.
		\item{}For any integers $n , m_1 , \dots , m_n \geq 0$, and operations $f, f'  \in \s{O}(n)$, $k_1 \in \s{O}(m_1)$, \dots, $k_n \in \s{O}(m_n)$, if $f R f'$, then $\gamma(f;k_1,\dots,k_n) R \gamma(f';k_1,\dots,k_n)$.
		\item{}For any integers $n , m_1 , \dots , m_n \geq 0$ and $1 \leq i \leq n$, and operations $f \in \s{O}(n)$, $k_1 \in \s{O}(m_1)$, \dots , $k_i,k_i' \in \s{O}(m_i)$,\dots , $k_n \in \s{O}(m_n)$, if $k_i R k_i'$, then $\gamma(f;k_1,\dots,k_i,\dots,k_n) R \gamma(f;k_1,\dots,k_i',\dots,k_m)$.
	\end{enumerate}
Then conditions (1), (2), and (3) hold.
\end{prop}

\begin{proof}First of all, if $f \in \s{O}(n)$ is unreduced and $(g,\sigma) \in G \times \Sigma_n$, then (ii) implies that $g f \sigma$ is unreduced, and conversely. Condition (3) follows.

Next, we prove condition (1). First, observe that if $f$ is reduced and $\ol{f}$ is a reduced form of $f$, then any chain $f = f_0 R f_1 R \cdots R f_n = \ol{f}$ must degenerate to $f = f_0 = \ol{f}$. Therefore $f$ is its own, unique reduced form.

Now we argue by induction on the complexity of $f \in \s{O}$. Suppose first that $c(f) = 0$. Then $f$ is reduced, because $R$ strictly reduces $c$ and $c$ is nonnegative. Therefore $f$ is its own, unique reduced form. Next, suppose inductively that every $f'$ with $c(f') \leq N$ has a unique reduced form, and assume $c(f) = N+1$. If $f$ is reduced, then we are done. If not, then there is $h \in \s{O}$ such that $f R h$, and since $N+1 = c(f) > c(h)$, the operation $h$ has a unique reduced form $\ol{h}$. We claim that $\ol{h}$ is also the unique reduced form of $f$. For suppose $\ol{f}$ is reduced and we have a chain $f = f_0 R f_1 R \cdots R f_n = \ol{f}$. We must show that $\ol{h} = \ol{f}$. The inequality $n > 0$ holds because $f$ is unreduced, and it follows that $\ol{f}$ is the unique reduced form of $f_1$. By (i), there are coterminal chains $h = h_0 R \cdots R h_m = k$ and $f_1 = h_0' R \cdots R h_{m'}' = k$, and the operation $k$ has a unique reduced form $\ol{k}$ because $N \geq c(h) \geq c(k)$. Concatenating the chains $h R \cdots R k$ and $f_1 R \cdots R k$ with a chain $k R \cdots R \ol{k}$ exhibits $\ol{k}$ as the unique reduced form of $h$ and $f_1$. Therefore $\ol{h} = \ol{k} = \ol{f}$, which proves that $f$ has a unique reduced form. Condition (1) follows by induction.

Finally, we prove condition (2) by giving an explicit description of the congruence relation generated by $R$. Let $\ol{(-)} : \s{O} \to \s{O}$ be the function that sends an operation $f$ to its unique reduced form $\ol{f}$, and declare $f \sim h$ if and only if $\ol{f} = \ol{h}$. Then $\sim$ is a graded equivalence relation. We shall show it is a congruence relation. Given any $n$-ary operation $f$ and $(g,\sigma) \in G \times \Sigma_n$, there is a chain $f = f_0 R \cdots R f_n = \ol{f}$, and applying (ii) gives another chain $(gf\sigma) = (g f_0 \sigma) R \cdots R (g f_n \sigma) = (g \ol{f} \sigma)$. Since $g \ol{f} \sigma$ is reduced, this shows that $\ol{g f \sigma} = g \ol{f} \sigma$. Thus, if $f \sim h$, then $\ol{g f \sigma} = g \ol{f} \sigma = g \ol{h} \sigma = \ol{g h \sigma}$, and hence $g f \sigma \sim g h \sigma$. Now suppose $f , h_1 , \dots , h_n$ are operations, where $n$ is the arity of $f$, and consider the composite $\gamma(f;h_1,\dots,h_n)$. There is a $R$-chain connecting $f$ to $\ol{f}$, and applying (iii) gives a $R$-chain from $\gamma(f;h_1,\dots,h_n)$ to $\gamma(\ol{f};h_1,\dots,h_n)$. Applying (iv) to the $R$-chains for $h_1,\dots,h_n$ and concatenating the results gives a $R$-chain from $\gamma(f;h_1,\dots,h_n)$ to $\gamma(\ol{f};\ol{h_1},\dots,\ol{h_n})$, which we then may concatenate with a chain from $\gamma(\ol{f};\ol{h_1},\dots,\ol{h_n})$ to $\ol{\gamma(\ol{f};\ol{h_1},\dots,\ol{h_n})}$. This shows that $\ol{\gamma(f;h_1,\dots,h_n)} = \ol{\gamma(\ol{f};\ol{h_1},\dots,\ol{h_n})}$. Thus, if $f \sim f'$, $h_1 \sim h_1'$, \dots, $h_n \sim h_n'$, then $\ol{\gamma(f;h_1,\dots,h_n)} = \ol{\gamma(\ol{f};\ol{h_1},\dots,\ol{h_n})} = \ol{\gamma(\ol{f'};\ol{h_1'},\dots,\ol{h_n'})} = \ol{\gamma(f';h_1',\dots,h_n')}$, and hence $\gamma(f;h_1,\dots,h_n) \sim \gamma(f';h_1',\dots,h_n')$. It follows that $\sim$ is a congruence relation.

Now suppose $\approx$ is any congruence relation on $\s{O}$ that contains $R$. Given any operation $f \in \s{O}$, the existence of a $R$-chain $f = f_0 R \cdots R f_n = \ol{f}$ implies that $f \approx \ol{f}$. Thus, if $f \sim h$, then $f \approx \ol{f} = \ol{h} \approx h$, and hence $f \approx h$. This shows that $\sim$ is the least congruence relation containing $R$. To find representatives for $\sim \,\, = \la R \ra$, note that if $f \in \s{O}$ is reduced, then $\ol{f} = f$. It follows that $\ol{f} \sim f$, because $\ol{\ol{f}} = \ol{f}$, and if $h$ is any reduced operation such that $h \sim f$, then $h = \ol{h} = \ol{f}$. Therefore every $f \in \s{O}$ is $\sim$-equivalent to a unique reduced operation, namely $\ol{f} \in r\s{O}$. This proves condition (2).
\end{proof}

One can often restrict the chains in (i) to equalities $h = k = h'$ or to individual $R$-relations $h R k$ and $h' Rk$, but we will use more general chains in our analysis of tensor products. The consequent clauses in (ii)--(iv) can also be weakened to allow for $R$-chains rather than just $R$-relations, but Propositions \ref{prop:idquot} and \ref{prop:idquotcrit} seem to apply as stated in many interesting examples. For example, we used them implicitly in \cite[\S8]{RubComb} to identify the associativity $N$ operad $\b{As}_{\c{T}}$, and the underlying nonequivariant operad of the free $G$-operad $F(S)$ on a $\Sigma$-free symmetric sequence $S$.

\subsection{Free operads}\label{subsec:freeop} In preparation for our analysis of coproducts and tensor products, we briefly recall a description of the free operad $F(S)$ on a $\Sigma$-free symmetric sequence $S$. The following is taken from \cite[\S\S7--8]{RubComb}.

Suppose $S$ is a $\Sigma$-free symmetric sequence in $\b{Set}^G$, and think of the elements $f \in S(n)$ as formal $n$-ary operations. Just as the free monoid on a set $X$ is a collection of formal products of elements of $X$, the free symmetric operad on $S$ is a collection of formal composites of operations in $S$. The wrinkle is that we can restrict our operations to a set of $\Sigma$-orbit representatives for $S$ by using the $\Sigma$-equivariance of composition.

For each $n \geq 0$, choose a set $\ol{S}(n) \subset S(n)$ of $\Sigma_n$-orbit representatives. Now consider formal words $w$, such that each letter of $w$ is either:
	\begin{enumerate}[label=(\alph*)]
		\item{}a element of $\ol{S}$,
		\item{}a variable symbol $x_i$ for some integer $i > 0$,
		\item{}a left or right parenthesis, or
		\item{}a comma.
	\end{enumerate}
Such a word is a \emph{term} if it is built at some stage of the following recursion:
	\begin{enumerate}
		\item{}every variable $x_i$ is a term, and
		\item{}if $f \in \ol{S}(n)$ and $t_1,\dots,t_n$ are terms, then $f(t_1,\dots,t_n)$ is also a term.
	\end{enumerate}
The \emph{arity} of a term is the number of distinct variable symbols that occur in it, and a $n$-ary term $t$ is \emph{operadic} if each of the variables $x_1,\dots,x_n$ appears exactly once in $t$. The $n$th level of the free operad $F(S)$ is the set of all $n$-ary operadic terms.

If $t$ is a $n$-ary operadic term and $\sigma \in \Sigma_n$, then the term $t \cdot \sigma$ is obtained by replacing each variable $x_i$ in $t$ with the variable $x_{\sigma^{-1}i}$.

If $t$ is a $k$-ary operadic term, and $s_i$ is a $j_i$-ary operadic term for $i = 1,\dots,k$, then the $(j_1 + \cdots + j_k)$-ary operadic term $\gamma(t;s_1,\dots,s_k)$ is obtained by adding $j_1 + \cdots + j_{i-1}$ to the subscript of each variable in $s_i$ -- call the result $s_i'$ -- and then substituting the (non-operadic) terms $s_1',\dots, s_k'$ for the variables $x_1,\dots,x_k$ in $t$. The term $x_1$ is the identity for $\gamma$.

The $G$-action on $F(S)$ is defined defined recursively. For any $g \in G$, declare
	\begin{enumerate}
		\item{}$g * x_i = x_i$ for every $i > 0$, and
		\item{}$g * f(t_1,\dots,t_n) = f'(g * t_{\sigma^{-1}1} , \dots , g * t_{\sigma^{-1}n})$, where $g \cdot f = f' \cdot \sigma$ for unique elements $f' \in \ol{S}(n)$ and $\sigma \in \Sigma_n$.
	\end{enumerate}

These data make $F(S)$ into a symmetric operad in $\b{Set}^G$, which is free on the $G$-symmetric sequence $S$. The unit map $\eta : S \to F(S)$ sends $f \in \ol{S}(n)$ to $f(x_1,\dots,x_n)$, and the rest is determined by $\Sigma$-equivariance.

\subsection{Coproducts and tensor products}\label{subsec:copten} This section analyzes two specific presentations of the coproduct of $N$ operads and the tensor product of free $G$-operads. We use Propositions \ref{prop:idquot} and \ref{prop:idquotcrit} to solve the associated word problems, and thus determine the underlying symmetric sequences of these operads.

We begin with coproducts. For motivation, suppose $F : \b{Set} \leftrightarrows \b{Grp} : U$ is the free-forgetful adjunction for nonabelian groups. Given $G, H \in \b{Grp}$, the coproduct $G * H$ may be constructed as a subset of the free group $F(UG \sqcup UH)$, equipped with a reduced concatenation product. This construction generalizes to operads. The next result is standard, but we include it as an example of how the assumptions in Proposition \ref{prop:idquotcrit} work.

\begin{exlem}\label{exlem:Nopcop} Suppose $\s{O}$ and $\s{P}$ are $N$ operads, and let $\s{O} * \s{P}$ be their coproduct in $\b{Op}^G$. Then $\s{O} * \s{P}$ is isomorphic to a sub-symmetric sequence of $F(U\s{O} \sqcup U \s{P})$, equipped with a modified composition operation.
\end{exlem}

\begin{proof} Suppose $F : \b{Sym}^G \rightleftarrows \b{Op}^G : U$ is the free-forgetful adjunction, and form the free operad $\s{F} = F(U\s{O} \sqcup U\s{P})$. Let $i : U\s{O} \hookrightarrow UFU\s{O} \to UF(U\s{O} \sqcup U\s{P})$ be the composite of the unit and the map induced by the inclusion $U\s{O} \hookrightarrow U\s{O} \sqcup U\s{P}$, and let $j : U\s{P} \to UF(U\s{O} \sqcup U\s{P})$ be defined similarly. Then $\s{O} * \s{P}$ is isomorphic to the quotient of $\s{F}$ by the congruence relation generated by
	\begin{align*}
		i(\t{id}_{\s{O}})  \,\, E \,\,  \t{id}_{\s{F}} \quad&,\quad i(h) \circ_k i(f) \,\, E \,\, i(h \circ_k f)	\\
		j(\t{id}_{\s{P}}) \,\, E \,\, \t{id}_{\s{F}} \quad&,\quad j(h) \circ_k j(f) \,\, E \,\, j(h \circ_k f) ,
	\end{align*}
where $\circ_k$ denotes partial composition, and the operations $h$ and $f$ are taken from $\s{O}$ in the first line and $\s{P}$ in the second line. We analyze this quotient using the model for $\s{F}$ described in \S\ref{subsec:freeop}.

Let $X$ and $Y$ be sets of $\Sigma$-orbit representatives for $\s{O}$ and $\s{P}$. Then the elements of $\s{F}$ are formal composites of operations in $X \sqcup Y$. Given two such composites $t$ and $t'$, declare $t R t'$ if we can obtain $t'$ from $t$ by replacing a subterm of $t$ in one of the following ways:
	\begin{enumerate}[label=(\alph*)]
		\item{}$\t{id}_{\s{O}}(t_1) \rightsquigarrow t_1$,
		\item{}$\t{id}_{\s{P}}(t_1) \rightsquigarrow t_1$,
		\item{}$h(t_1,\dots,f(t_k,\dots,t_{k+\abs{f}-1}),\dots,t_{\abs{h} + \abs{f} - 1}) \rightsquigarrow \ell (t_{\sigma^{-1}1}, \dots , t_{\sigma^{-1}(\abs{h} + \abs{f} - 1)})$, where $h,f \in X$ and $h \circ_k f = \ell \cdot \sigma$ for $\ell  \in X$ and $\sigma \in \Sigma_{\abs{h} + \abs{f} - 1}$, or
		\item{}$h(t_1,\dots,f(t_k,\dots,t_{k+\abs{f}-1}),\dots,t_{\abs{h} + \abs{f} - 1}) \rightsquigarrow \ell (t_{\sigma^{-1}1}, \dots , t_{\sigma^{-1}(\abs{h} + \abs{f} - 1)})$, where $h,f \in Y$ and $h \circ_k f = \ell \cdot \sigma$ for $\ell \in Y$ and $\sigma \in \Sigma_{\abs{h} + \abs{f} - 1}$.
	\end{enumerate}
Here, the $t_i$ are not necessarily operadic terms. The relation $R$ generates the same congruence relation as $E$.

For any formal composite $t \in \s{F}$, let $c(t)$ be the number of operation symbols in $t$. Then $R$ reduces $c$. The relation $R$ satisfies conditions (iii) and (iv) in Proposition \ref{prop:idquotcrit} because it is defined in terms of substitutions of subterms. It satisfies the $\Sigma$-equivariance portion of (ii) for the same reason. As for the $G$-equivariance, if $g \in G$, then mutliplication $g * (-)$ preserves (a)-substitutions because $\t{id}_{\s{O}}$ is $G$-fixed, and similarly for (b)-substitutions. For (c)-substitutions, we use the $G$-operad axioms and the $\Sigma$-freeness of $\s{F}$. Suppose $g \cdot h = h' \cdot \tau$ and $g \cdot f = f' \cdot \upsilon$, for $h',f' \in X$ and permutations $\tau, \upsilon$. Then $g * h(t_1,\dots,f(t_k,\dots,t_{k+\abs{f}-1}),\dots,t_{\abs{h} + \abs{f} - 1})$ equals
	\[
	h'(s_1,\dots,f'(s_{\tau k}, \dots, s_{\tau k + \abs{f} - 1}),\dots,s_{\abs{h} + \abs{f} -1})
	\]
where $s_i = g * t_{\alpha^{-1} i}$ for the permutation $\alpha = \tau(1,\dots,\abs{f},\dots,1) \cdot (\t{id} \sqcup \cdots \sqcup \upsilon \sqcup \cdots \sqcup \t{id})$. If $h' \circ_{\tau k} f' = m \cdot \pi$ for $m \in X$ and $\pi$ a permutation, then
	\[
	h'(s_1,\dots,f'(s_{\tau k}, \dots, s_{\tau k + \abs{f} - 1}),\dots,s_{\abs{h} + \abs{f} -1}) \rightsquigarrow m(s_{\pi^{-1}1}, \dots, s_{\pi^{-1}(\abs{h} + \abs{f} - 1)})
	\]
is a (c)-substitution. On the other hand, suppose $g \cdot \ell = \ell' \cdot \nu$ for $\ell' \in X$ and a permutation $\nu$. Then
	\[
	g * \ell(t_{\sigma^{-1}1},\dots,t_{\sigma^{-1}(\abs{h} + \abs{f} - 1)}) = \ell'(g * t_{\sigma^{-1}\nu^{-1}1},\dots, g * t_{\sigma^{-1}\nu^{-1}(\abs{h} + \abs{f} -1)}).
	\]
We claim that $m = \ell'$ and $s_{\pi^{-1}i} = g* t_{\sigma^{-1}\nu^{-1}i}$. To see this, note that
	\[
	m \cdot \pi \cdot \alpha = (g \cdot h) \circ_k (g \cdot f) = g \cdot (h \circ_k f) = g \cdot \ell \cdot \sigma = \ell' \cdot \nu \cdot \sigma.
	\]
Now, since $m,\ell' \in X$ are $\Sigma$-orbit representatives and $\s{O}$ is $\Sigma$-free, it follows that $m = \ell'$ and $\pi \cdot \alpha = \nu \cdot \sigma$, and therefore $s_{\pi^{-1}i} = g * t_{\alpha^{-1}\pi^{-1}i} = g * t_{\sigma^{-1} \nu^{-1}i}$. Thus $g * (-)$ preserves (c)-substitutions, and similarly for (d)-substitutions. This shows that $R$ satisfies condition (ii) of Proposition \ref{prop:idquotcrit}.

To verify that $R$ satisfies condition (i), we must analyze the degree to which substitutions (a)--(d) commute. There are 10 cases to consider, but most are uninteresting. For example, suppose $t R s$ via an (a)-substitution $\t{id}_{\s{O}}(t_1) \rightsquigarrow t_1$ and $t R s'$ via a (b)-substitution $\t{id}_{\s{P}}(t_1) \rightsquigarrow t_1$. Then these substitutions are disjoint, in the sense that they either occur in non-overlapping subwords of $t$, or one substitution occurs inside the $t_1$-term of the other. Thus, we obtain a term $r$ such that $s R r$ and $s' R r$ via the complementary substitutions. Similar reasoning applies for paired ((a),(d)), ((b),(c)), and ((c),(d))-substitutions. Likewise, if $t R s$ and $t R s'$ via two (a)-substitutions, then either we have made the same substitution and $s = s'$, or they are disjoint and there is a term $r$ such that $s R r$ and $s' R r$. Similarly for pairs of (b)-substitutions.

The interesting cases are those in which the substitutions can interact. For example, suppose $t R s$ and $t R s'$ via non-disjoint (a) and (c)-substitutions. Then either $h = \t{id}_{\s{O}}$ or $f = \t{id}_{\s{O}}$, and the operadic identity axiom implies that $s = s'$. Therefore condition (i) of Proposition \ref{prop:idquotcrit} holds for paired (a) and (c)-substitutions. Similarly for paired (b) and (d)-substitutions.

Now suppose that $t R s$ and $t R s'$ via unequal and non-disjoint (c)-substitutions
	\[
	h_i(t_1,\dots,f_i(t_k,\dots,t_{k+\abs{f}-1}),\dots,t_{\abs{h} + \abs{f} - 1}) \rightsquigarrow \ell_i (t_{\sigma^{-1}1}, \dots , t_{\sigma^{-1}(\abs{h} + \abs{f} - 1)})
	\]
for $i = 1,2$. There are three possibilities: either $h_2 = f_1$, or $h_1 = f_2$, or $h_1 = h_2$ and $f_1$ and $f_2$ occur in different positions. In any case, by using the associativity and $\Sigma$-equivariance axioms for partial composition, we can find an $r$ such that $s R r$ and $s' R r$. Therefore condition (i) holds for pairs of (c)-substitutions, and similarly for pairs of (d)-substitutions.

By Proposition \ref{prop:idquotcrit}, we deduce that conditions (1)--(3) of \S\ref{subsec:solveword} hold for the $R$-reduced operations in $\s{F}$. Since $\s{F}$ is a $N$ operad, Proposition \ref{prop:idquot} implies that $\s{O} * \s{P} \cong \s{F}/\la E \ra = \s{F}/ \la R \ra$ is isomorphic to the sub-symmetric sequence of $R$-reduced operations in $\s{F}$, equipped with a modified composition operation.
\end{proof}

Lastly, we consider tensor products of free operads.

\begin{lem}\label{lem:freeopten} Suppose $S$ and $T$ are $\Sigma$-free symmetric sequences in $\b{Set}^G$ such that the sets $S(n)^G$ and $T(n)^G$ are nonempty for all $n \geq 0$. Then the tensor product $F(S) \otimes F(T)$ is isomorphic to a sub-symmetric sequence of $F(S \sqcup T)$, equipped with a modified composition operation.
\end{lem}

\begin{proof} The tensor product $F(S) \otimes F(T)$ is the coproduct $F(S) * F(T) \cong F(S \sqcup T)$ modulo the congruence relation $\sim$ generated by the set of interchange relations. We shall analyze it using a different set of relations to avoid issues that arise from nullary interchanges. Let $X$ and $Y$ be sets of $\Sigma$-orbit representatives for $S$ and $T$, and choose an element $z \in T(0)^G \subset Y(0)$. As in \S\ref{subsec:freeop}, we model $F(S \sqcup T)$ as a collection of formal composites of operations in $X \sqcup Y$. Given two elements $t , t' \in F(S \sqcup T)$, declare $tRt'$ if we can obtain $t'$ from $t$ by replacing a subterm of $t$ in one of the following ways: 
	\begin{enumerate}[label=(\alph*)]
		\item{}{ $h(f(t_{11},\dots,t_{1n}),\dots, f(t_{m1},\dots,t_{mn})) \rightsquigarrow 
				f(h(t_{11},\dots,t_{m1}),\dots,h(t_{1n},\dots,t_{mn}))$} for $h \in X(m)$ and $f \in Y(n)$ with $m,n > 0$, or
		\item{}one of the substitutions below, for $h \in X(m)$ and $f \in Y(n)$ with $m,n > 0$:
		{\small
		\[
		h(z(),f(t_{21},\dots,t_{2n}),\dots, f(t_{m1},\dots,t_{mn})) \rightsquigarrow f(h(z(),t_{21},\dots,t_{m1}),\dots,h(z(),t_{2n},\dots,t_{mn}))
		\]
		}
		and similarly for $z()$'s in the $2$nd through $m$th positions, or
		{\small
		\[
		h(z(),z(),f(t_{31},\dots,t_{3n}),\dots, f(t_{m1},\dots,t_{mn})) \rightsquigarrow f(h(z(),z(),t_{31},\dots,t_{m1}),\dots,h(z(),z(),t_{3n},\dots,t_{mn}))
		\]
		}
		and similarly for pairs of $z()$'s in any other positions, or
		\[
		\vdots
		\]
		{\small
		\[
		h(z(),\dots,z(),f(t_{m1},\dots,t_{mn})) \rightsquigarrow f(h(z(),\dots,z(),t_{m1}),\dots,h(z(),\dots,z(),t_{mn}))
		\]
		}
		and similarly for $m-1$ copies of $z()$ in any other positions, or
		\item{}$\ell (z(),z(),\dots,z()) \rightsquigarrow z()$ for $\ell \in X(n) \sqcup Y(n)$ and $n > 0$, or
		\item{}$e() \rightsquigarrow z()$ for $e \in [ X(0) \sqcup Y(0) ] \setminus \{z\}$.
	\end{enumerate}
As before, the $t_i$ are not necessarily operadic terms. The congruence relation $\sim$ contains all (a)-substitutions by design. Then, since the operad $F(S) \otimes F(T) \cong F(S \sqcup T)/\!\sim$ is reduced, it follows that $\sim$ also contains all (c) and (d)-substitutions. The (b)-substitutions are generated by (a) and (c)-substitutions. Therefore $\la R \ra \subset \,\, \sim$. To establish the other inclusion, it is enough to show that $\la R \ra$ contains all interchange relations
	\[
	h(f(x_{11},\dots,x_{1n}),\dots,f(x_{m1},\dots,x_{mn}))  \sim  f(h(x_{11},\dots,x_{m1}) , \dots , h(x_{1n},\dots,x_{mn})),
	\]
where $h \in X$ and $f \in Y$ are possibly nullary. If $h$ and $f$ have positive arity, this is an (a)-substitution. If either $h$ or $f$ is nullary, then both sides of the relation are nullary operations in $F(S \sqcup T)$. The operad $F(S \sqcup T)/\la R \ra$ is reduced, because (c) and (d)-substitutions can reduce any nullary operation to $z()$. Therefore all nullary interchange relations are also contained in $\la R \ra$, and therefore $\sim \,\, \subset \la R \ra$.

For any $t \in F(S \sqcup T)$ and operation symbol $f$ in $t$, define the \emph{depth} $d(f)$ of $f$ to be the number of nested pairs of parentheses that contain $f$. For example, the $f$ in $f(x_1,\dots,x_n)$ has depth $0$, while the $f$ in $h(f(x_1,x_2),k(x_3,x_4))$ has depth $1$. We define the complexity of a term $t$ by
	\[
	c(t) = \#\Bigg(
	\begin{array}{c}
	\t{operation symbols in $t$} \\
	\t{not equal to $z$}
	\end{array}
	\Bigg)
	\quad\quad+\!\!
	\sum_{\tiny \begin{array}{c}
	\t{$Y$-operation}\\
	\t{symbols $f$ in $t$}
	\end{array}}
	\!\!
	d(f) \cdot \abs{f} ,
	\]
where $\abs{f}$ denotes the arity of $f$. The relation $R$ reduces this complexity function. In (a) and (b)-substitutions, the right summand decreases by at least $\abs{f}$, but the left summand cannot increase by more than $\abs{f} - 1$. In (c) and (d)-substitutions, the left summand decreases and the right summand does not increase.

The relation $R$ satisfies conditions (ii), (iii), and (iv) of Proposition \ref{prop:idquotcrit} by the same arguments used in Example-Lemma \ref{exlem:Nopcop}. To verify condition (i), we consider the possible interactions between substitutions. Suppose $t R s$ and $t R s'$. In almost all cases, the substitutions that yield $s$ and $s'$ must either be equal or disjoint, in which case $s = s'$ or there is a term $r$ such that $sRr$ and $s'Rr$ via the complementary substitutions. The only interesting scenario is when $s$ is obtained by an (a) or (b)-substitution, $s'$ is obtained by a (c)-substitution, and the subterm $\ell(z(),\dots,z())$ for the (c)-substitution is equal to one of the $f(t_{i1},\dots,t_{in})$ blocks in the (a) or (b)-substitution. Suppose for simplicity that $s$ is obtained by an (a)-substitution, and the block $f(t_{11},\dots,t_{1n})$ equals $\ell(z(),\dots,z())$. Then $s' R s$ via the very first (b)-substitution. The same reasoning applies when $\ell(z(),\dots,z())$ is another block, or if $s$ is obtained by a (b)-substitution that contains at least two $f(t_{i1},\dots,t_{in})$-blocks on the left side. If $s$ is obtained from a (b)-substitution with only one $f(t_{i1},\dots,t_{in})$-block, then collapsing the entire subterm down to $z()$ yields a term $r$ such that $s' R r$ and $s R \cdots R r$ via (c)-substitutions. Therefore $R$ satisfies condition (i), and Propositions \ref{prop:idquotcrit} and \ref{prop:idquot} identify the tensor product $F(S) \otimes F(T) \cong F(S \sqcup T)/ \!\sim \,\, = F( S \sqcup T)/\la R \ra$ with the sub-symmetric sequence of $R$-reduced operations in $F(S \sqcup T)$, equipped with a modified composition operation.
\end{proof}

\end{document}